\documentclass{article}
\usepackage[a4paper,top=3cm,bottom=2cm,left=3.5cm,right=3.5cm,marginparwidth=4cm]{geometry}
\usepackage{graphicx} 
\usepackage{amsmath}
\usepackage{amsthm}
\usepackage{todonotes}
\usepackage{amssymb}
\usepackage{mathabx}
\usepackage{tikz}
\usepackage{subcaption}
\usepackage{forest}
\usepackage[style=numeric, sorting = nyt,backend=biber, maxbibnames=99]{biblatex}
\usepackage{framed}
\usepackage{tabularx}
\usepackage{xspace}
\usepackage{thm-restate}
\usepackage{authblk}
\usepackage{lineno}
\usepackage{hyperref}
\usepackage{cleveref}

\newcommand{\yuki}[1]{\textcolor{black}{#1}}
\newcommand{\taka}[1]{\textcolor{black}{#1}}
\newcommand{\balint}[1]{\textcolor{black}{#1}}

\newcommand{\rev}[1]{\textcolor{black}{#1}}

\usepackage{tikz}

\newcommand{\treePreliminaries}[1]{
\begin{forest}
for tree={l sep=4mm*#1, s sep=6mm*#1, edge={-}, anchor=center}
[{\,}, circle, draw=black, fill=black, inner sep=0pt, minimum size=6pt
    [{\,}, circle, draw=black, fill=black, inner sep=0pt, minimum size=6pt
        [{\,}, circle, draw=black, fill=black, inner sep=0pt, minimum size=6pt, label=below:1, tier=leaf] 
        [{\,}, circle, draw=black, fill=black, inner sep=0pt, minimum size=6pt, label=below:2, tier=leaf] 
    ]
    [{\,}, circle, draw=black, fill=black, inner sep=0pt, minimum size=6pt
        [{\,}, circle, draw=black, fill=black, inner sep=0pt, minimum size=6pt
            [{\,}, circle, draw=black, fill=black, inner sep=0pt, minimum size=6pt, label=below:3, tier=leaf] 
            [{\,}, circle, draw=black, fill=black, inner sep=0pt, minimum size=6pt, label=below:4, tier=leaf] 
        ]
        [{\,}, circle, draw=black, fill=black, inner sep=0pt, minimum size=6pt, label=below:5, tier=leaf] 
    ]
]
\end{forest}
}

\newcommand{\cherry}[3]{
\begin{forest}
for tree={l sep=4mm*#1, s sep=6mm*#1, edge={-}, anchor=center}
[{\,}, circle, draw=black, fill=black, inner sep=0pt, minimum size=6pt
    [{\,}, circle, draw=black, fill=black, inner sep=0pt, minimum size=6pt, label=below:#2, tier=leaf] 
    [{\,}, circle, draw=black, fill=black, inner sep=0pt, minimum size=6pt, label=below:#3, tier=leaf] 
]
\end{forest}
}

\newcommand{\treeMultiplicationLemma}[1]{
\begin{forest}
for tree={l sep=4mm*#1, s sep=10mm*#1, edge={-}, anchor=center}
[{\,}, circle, draw=black, fill=black, inner sep=0pt, minimum size=6pt
    [{\,}, circle, draw=black, fill=black, inner sep=0pt, minimum size=6pt
        [{\,}, circle, draw=black, fill=black, inner sep=0pt, minimum size=6pt, label=right:$v$
            [{\,}, circle, draw=black, fill=black, inner sep=0pt, minimum size=6pt, label=below:1, tier=leaf]
            [{\,}, circle, draw=black, fill=black, inner sep=0pt, minimum size=6pt, label=below:2, tier=leaf]
        ]
        [{\,}, circle, draw=black, fill=black, inner sep=0pt, minimum size=6pt, tier=leaf, label=below:3]
    ]
    [{\,}, circle, draw=black, fill=black, inner sep=0pt, minimum size=6pt
        [{\,}, circle, draw=black, fill=black, inner sep=0pt, minimum size=6pt, label=below:4, tier=leaf]
        [{\,}, circle, draw=black, fill=black, inner sep=0pt, minimum size=6pt, label=below:5, tier=leaf]
    ]
]
\end{forest}
}

\newcommand{\treeMultiplicationA}[1]{
\begin{forest}
for tree={l sep=4mm*#1, s sep=10mm*#1, edge={-}, anchor=center}
[{\,}, circle, draw=black, fill=black, inner sep=0pt, minimum size=6pt
    [{\,}, circle, draw=black, fill=black, inner sep=0pt, minimum size=6pt
        [{\,}, circle, draw=black, fill=black, inner sep=0pt, minimum size=6pt, label=below:1, tier=leaf]
        [{\,}, circle, draw=black, fill=black, inner sep=0pt, minimum size=6pt, label=below:3, tier=leaf]
    ]
    [{\,}, circle, draw=black, fill=black, inner sep=0pt, minimum size=6pt, label=below:4, tier=leaf]
]
\end{forest}
}

\newcommand{\treeMultiplicationB}[1]{
\begin{forest}
for tree={l sep=4mm*#1, s sep=10mm*#1, edge={-}, anchor=center}
[{\,}, circle, draw=black, fill=black, inner sep=0pt, minimum size=6pt
    [{\,}, circle, draw=black, fill=black, inner sep=0pt, minimum size=6pt, label=below:2, tier=leaf]
    [{\,}, circle, draw=black, fill=black, inner sep=0pt, minimum size=6pt, label=below:4, tier=leaf]
]
\end{forest}
}

\newcommand{\treeSequence}[1]{
\begin{forest}
for tree={l sep=4mm*#1, s sep=8mm*#1, edge={-}, anchor=center}
[{\,}, circle, draw=black, fill=black, inner sep=0pt, minimum size=6pt
    [{\,}, circle, draw=black, fill=black, inner sep=0pt, minimum size=6pt
        [{\,}, circle, draw=black, fill=black, inner sep=0pt, minimum size=6pt
            [{\,}, circle, draw=black, fill=black, inner sep=0pt, minimum size=6pt, label=below:1, tier=leaf]
            [{\,}, circle, draw=black, fill=black, inner sep=0pt, minimum size=6pt, label=below:2, tier=leaf]
        ]
        [{\,}, circle, draw=black, fill=black, inner sep=0pt, minimum size=6pt, tier=leaf, label=below:3]
    ]
    [{\,}, circle, draw=black, fill=black, inner sep=0pt, minimum size=6pt
        [{\,}, circle, draw=black, fill=black, inner sep=0pt, minimum size=6pt, label=below:4, tier=leaf]
        [{\,}, circle, draw=black, fill=black, inner sep=0pt, minimum size=6pt, label=below:5, tier=leaf]
    ]
]
\end{forest}
}

\newcommand{\binaryRefinementOriginal}[1]{
\begin{forest}
for tree={l sep=2mm*#1, s sep=8mm*#1, edge={-}, anchor=center}
[{\,}, circle, draw=black, fill=black, inner sep=0pt, minimum size=6pt, label=$\rho$
    [{\,}, circle, draw=black, fill=black, inner sep=0pt, minimum size=6pt, label =$a$
        [{\,}, circle, draw=black, fill=black, inner sep=0pt, minimum size=6pt, label=below:1, tier=leaf]
        [{\,}, circle, draw=black, fill=black, inner sep=0pt, minimum size=6pt, label=below:2, tier=leaf]
    ]
    [{\,}, circle, draw=black, fill=black, inner sep=0pt, minimum size=6pt, label=below:3, tier=leaf]
    [{\,}, circle, draw=black, fill=black, inner sep=0pt, minimum size=6pt, label=$b$
        [{\,}, circle, draw=black, fill=black, inner sep=0pt, minimum size=6pt, label=below:4, tier=leaf]
        [{\,}, circle, draw=black, fill=black, inner sep=0pt, minimum size=6pt, label=below:5, tier=leaf]
    ]
    [{\,}, circle, draw=black, fill=black, inner sep=0pt, minimum size=6pt, label=below:6, tier=leaf]
    [{\,}, circle, draw=black, fill=black, inner sep=0pt, minimum size=6pt, label = $c$
        [{\,}, circle, draw=black, fill=black, inner sep=0pt, minimum size=6pt, label=below:7, tier=leaf]
        [{\,}, circle, draw=black, fill=black, inner sep=0pt, minimum size=6pt, label=below:8, tier=leaf]
    ]
]
\end{forest}
}

\newcommand{\biniaryRefinementOptimal}[1]{
\begin{forest}
for tree={l sep=4mm*#1, s sep=8mm*#1, edge={-}, anchor=center}
[{\,}, circle, draw=black, fill=black, inner sep=0pt, minimum size=6pt, label=left:$\rho_1$
    [{\,}, circle, draw=black, fill=black, inner sep=0pt, minimum size=6pt, label=left:$\rho_2$
        [{\,}, circle, draw=black, fill=black, inner sep=0pt, minimum size=6pt, label=left:$\rho_3$
            [{\,}, circle, draw=black, fill=black, inner sep=0pt, minimum size=6pt, label=left:$\rho_4$
                [{\,}, circle, draw=black, fill=black, inner sep=0pt, minimum size=6pt
                    [{\,}, circle, draw=black, fill=black, inner sep=0pt, minimum size=6pt, label=below:1, tier=leaf]
                    [{\,}, circle, draw=black, fill=black, inner sep=0pt, minimum size=6pt, label=below:2, tier=leaf]
                ]
                [{\,}, circle, draw=black, fill=black, inner sep=0pt, minimum size=6pt
                    [{\,}, circle, draw=black, fill=black, inner sep=0pt, minimum size=6pt, label=below:4, tier=leaf]
                    [{\,}, circle, draw=black, fill=black, inner sep=0pt, minimum size=6pt, label=below:5, tier=leaf]
                ]
            ]
            [{\,}, circle, draw=black, fill=black, inner sep=0pt, minimum size=6pt
                [{\,}, circle, draw=black, fill=black, inner sep=0pt, minimum size=6pt, label=below:7, tier=leaf]
                [{\,}, circle, draw=black, fill=black, inner sep=0pt, minimum size=6pt, label=below:8, tier=leaf]
            ]
        ]
        [{\,}, circle, draw=black, fill=black, inner sep=0pt, minimum size=6pt, label=below:3, tier=leaf]
    ]
    [{\,}, circle, draw=black, fill=black, inner sep=0pt, minimum size=6pt, label=below:6, tier=leaf]
]
\end{forest}
}

\newcommand{\binaryRefinementSuboptimal}[1]{
\begin{forest}
for tree={l sep=4mm*#1, s sep=8mm*#1, edge={-}, anchor=center}
[{\,}, circle, draw=black, fill=black, inner sep=0pt, minimum size=6pt
    [{\,}, circle, draw=black, fill=black, inner sep=0pt, minimum size=6pt
        [{\,}, circle, draw=black, fill=black, inner sep=0pt, minimum size=6pt
            [{\,}, circle, draw=black, fill=black, inner sep=0pt, minimum size=6pt
                [{\,}, circle, draw=black, fill=black, inner sep=0pt, minimum size=6pt
                    [{\,}, circle, draw=black, fill=black, inner sep=0pt, minimum size=6pt, label=below:1, tier=leaf]
                    [{\,}, circle, draw=black, fill=black, inner sep=0pt, minimum size=6pt, label=below:2, tier=leaf]
                ]
                [{\,}, circle, draw=black, fill=black, inner sep=0pt, minimum size=6pt, label=below:3, tier=leaf]
            ]
            [{\,}, circle, draw=black, fill=black, inner sep=0pt, minimum size=6pt, label=below:4, tier=leaf]
        ]
        [{\,}, circle, draw=black, fill=black, inner sep=0pt, minimum size=6pt
            [{\,}, circle, draw=black, fill=black, inner sep=0pt, minimum size=6pt, label=below:5, tier=leaf]
            [{\,}, circle, draw=black, fill=black, inner sep=0pt, minimum size=6pt, label=below:6, tier=leaf]
        ]
    ]
    [{\,}, circle, draw=black, fill=black, inner sep=0pt, minimum size=6pt
        [{\,}, circle, draw=black, fill=black, inner sep=0pt, minimum size=6pt, label=below:7, tier=leaf]
        [{\,}, circle, draw=black, fill=black, inner sep=0pt, minimum size=6pt, label=below:8, tier=leaf]
    ]
]
\end{forest}
}

\newcommand{\doubleStarSix}[7]{%
\begin{tikzpicture}[
    scale=#7,
    every node/.style={circle,draw=black,fill=black,minimum size=4pt,inner sep=0pt}
]

\node[label={[yshift=-2pt]below:{#1}}] (A) at (0,0) {};
\node[label={[yshift=-2pt]below:{#4}}] (D) at (1.5,0) {};

\node[label={[yshift=-2pt]below:{#2}}] (B) at (-0.9,0.9) {};
\node[label={[yshift=-2pt]below:{#3}}] (C) at (-0.9,-0.9)  {};

\node[label={[yshift=-2pt]below:{#5}}] (E) at (2.4,0.9)  {};
\node[label={[yshift=-2pt]below:{#6}}] (F) at (2.4,-0.9)  {};

\draw (B)--(A);
\draw (C)--(A);
\draw (A)--(D);
\draw (D)--(E);
\draw (D)--(F);

\end{tikzpicture}
}

\newcommand{\pathFour}[5]{%
\begin{tikzpicture}[
    scale=#5,
    every node/.style={circle,draw=black,fill=black,minimum size=4pt,inner sep=0pt}
]

\node[label={[yshift=-2pt]below:{#1}}] (A) at (0,0) {};
\node[label={[yshift=-2pt]below:{#2}}] (B) at (1,0) {};

\node[label={[yshift=-2pt]below:{#3}}] (C) at (2,0) {};
\node[label={[yshift=-2pt]below:{#4}}] (D) at (3,0)  {};

\draw (A)--(B);
\draw (B)--(C);
\draw (C)--(D);

\end{tikzpicture}
}

\newcommand{\subconFive}[6]{%
\begin{tikzpicture}[
    scale=#6,
    every node/.style={circle,draw=black,fill=black,minimum size=4pt,inner sep=0pt}
]

\node[label={[yshift=-2pt]below:{#1}}] (A) at (0,0) {};
\node[label={[yshift=-2pt]below:{#2}}] (B) at (1,0) {};

\node[label={[yshift=-2pt]below:{#3}}] (C) at (2,0) {};
\node[label={[yshift=-2pt]below:{#4}}] (D) at (3,0)  {};

\node[label={[yshift=-2pt]below:{#5}}] (E) at (4,0) {};

\draw (A)--(B);
\draw (B)--(C);
\draw (C)--(D);
\draw (D)--(E);

\end{tikzpicture}
}

\newcommand{\starTreeFive}[1]{
\begin{forest}
for tree={l sep=4mm*#1, s sep=8mm*#1, edge={-}, anchor=center}
[{\,}, circle, draw=black, fill=black, inner sep=0pt, minimum size=4pt, label=above:5
    [{\,}, circle, draw=black, fill=black, inner sep=0pt, minimum size=4pt, label=below:1]
    [{\,}, circle, draw=black, fill=black, inner sep=0pt, minimum size=4pt, label=below:2]
    [{\,}, circle, draw=black, fill=black, inner sep=0pt, minimum size=4pt, label=below:3]
    [{\,}, circle, draw=black, fill=black, inner sep=0pt, minimum size=4pt, label=below:4]
]
\end{forest}
}

\newcommand{\coveringSequencesOriginal}{%
\begin{center}
\begin{tabularx}{\linewidth}{
|>{\centering\arraybackslash\hsize=.6\hsize}X
|>{\centering\arraybackslash\hsize=1.4\hsize}X|}
\hline
$S_1$ & $(2,1)(3,1)$ \\
$S_2$ & $(1,2)(3,2)$ \\
$S_3$ & $(1,2)(2,3)$ \\
$R_1$ & $(5,4)$ \\
$R_2$ & $(4,5)$ \\
\hline
\end{tabularx}
\end{center}
}

\newcommand{\coveringSequencesCombined}{%
\begin{center}
\begin{tabularx}{\linewidth}{
|>{\centering\arraybackslash\hsize=.4\hsize}X
|>{\centering\arraybackslash\hsize=1.6\hsize}X|}
\hline
$D_{1,1}$ & $(2,1)(3,1)(5,4)(1,4)$ \\
$D_{1,2}$ & $(2,1)(3,1)(4,5)(5,1)$ \\
$D_{2,1}$ & $(1,2)(3,2)(5,4)(4,2)$ \\
$D_{2,2}$ & $(1,2)(3,2)(4,5)(2,5)$ \\
$D_{3,1}$ & $(1,2)(2,3)(5,4)(4,3)$ \\
$D_{3,2}$ & $(1,2)(2,3)(4,5)(5,3)$ \\
\hline
\end{tabularx}
\end{center}
}

\newcommand{\tcnCounterexample}{%
\begin{figure}
    \centering
    \begin{subfigure}[b]{0.49\textwidth}
    \centering
    \begin{tikzpicture}[every node/.style = {draw, circle, fill, inner sep = 0pt, minimum size = 2mm},
square/.style = {regular polygon, regular polygon sides = 4, minimum size = 3 mm}]
    \tikzset {edge/.style = {very thick, shorten >= -0.5 pt}}

        \node[] (9) at (-1,-1.5) {};
        \node[] (10) at (-1.25,-2.0) {};

        \node[draw=none, fill=none, left = 5mm of 9] {\large{$T$}};

        \node[] (1) at (-1.5,-2.5) {};
        \node[draw=none, fill=none, below=1mm of 1] (leaf_1) {\large $1$};
        \node[] (2) at (-1.0,-2.5) {};
        \node[draw=none, fill=none, below=1mm of 2] (leaf_2) {\large $2$};
        \node[] (3) at (-0.5,-2.5) {};
        \node[draw=none, fill=none, below=1mm of 3] (leaf_3) {\large $3$};

        \draw[edge] (10) edge (1);
        \draw[edge] (9) edge (10);
        \draw[edge] (10) edge (2);
        \draw[edge] (9) edge (3);
    \end{tikzpicture}
    \end{subfigure}%
    \hfill
    \begin{subfigure}[b]{0.49\textwidth}
    \centering
    \begin{tikzpicture}[every node/.style = {draw, circle, fill, inner sep = 0pt, minimum size = 2mm},
square/.style = {regular polygon, regular polygon sides = 4, minimum size = 3 mm}]
    \tikzset {edge/.style = {very thick, shorten >= -0.5 pt}}

        \node[] (8) at (-0.5,-1.0) {};
        \node[] (9) at (-0.75,-1.5) {};
        \node[] (11) at (-0.25,-1.5) {};
        \node[] (10) at (-1,-2.25) {};
        \node[square] (13) at (0,-2.75) {};
        \node[] (14) at (0,-3.25) {};
        \node[square] (15) at (-1.25,-3.5) {};
        \node[] (16) at (-1.2, -2.8) {};

        \node[draw=none, fill=none, left = 5mm of 8] {\large{$N$}};

        \node[] (1) at (-1.5,-1.75) {};
        \node[draw=none, fill=none, below=1mm of 1] (leaf_1) {\large $1$};
        \node[] (2) at (-1.25,-4.0) {};
        \node[draw=none, fill=none, below=1mm of 2] (leaf_2) {\large $2$};
        \node[] (3) at (0,-4.0) {};
        \node[draw=none, fill=none, below=1mm of 3] (leaf_3) {\large $3$};
        \node[] (4) at (0.5,-1.75) {};
        \node[draw=none, fill=none, below=1mm of 4] (leaf_4) {\large $4$};
        \node[] (5) at (-2,-4.0) {};
        \node[draw=none, fill=none, below=1mm of 5] (leaf_5) {\large $5$};

        \draw[edge, dotted] (8) edge (9);
        \draw[edge, dotted] (8) edge (11);
        \draw[edge, dotted] (9) edge (1);
        \draw[edge, dotted] (9) edge (10);
        \draw[edge, dotted] (10) edge (16);
        \draw[edge] (10) edge (13);
        \draw[edge, dotted] (15) edge (2);
        \draw[edge, dotted] (13) edge (14);
        \draw[edge] (14) edge (15);
        \draw[edge, dotted] (16) edge (15);
        \draw[edge, dotted] (14) edge (3);
        \draw[edge] (11) edge (4);
        \draw[edge] (16) edge (5);
        \draw[edge, dotted] (11) edge (13);
    \end{tikzpicture}
    \end{subfigure}
    \caption{The tree~$T$ is a subnetwork of a tree-child network~$N$, and one embedding of $T$ into $N$ is shown (dotted edges).
    There exists no minimal tree-child sequence of $N$ that reduces $T$ to a single leaf (\Cref{obs:ctc(N)infinite}). For a full explanation, see \Cref{sec:TCOrch}.}
    \label{fig:TCNTreeContainmentCounterExample}
\end{figure}
}

\newcommand{\cpnCounterexample}{
\begin{figure}
    \centering
    \begin{subfigure}[b]{0.49\textwidth}
    \centering
    \begin{tikzpicture}[every node/.style = {draw, circle, fill, inner sep = 0pt, minimum size = 2mm},
square/.style = {regular polygon, regular polygon sides = 4, minimum size = 3 mm}]
    \tikzset {edge/.style = {very thick, shorten >= -0.5 pt}}

        \node[] (0) at (0.0,-0.5) {};
        \node[] (8) at (-0.5,-1.0) {};
        \node[] (11) at (0.5,-1.0) {};
        \node[] (9) at (-1,-1.5) {};
        \node[] (12) at (0.5,-1.5) {};
        \node[] (10) at (-0.75,-2.0) {};

        \node[draw=none, fill=none, left = 5mm of 0] {\large{$T$}};

        \node[] (1) at (-1.5,-2.5) {};
        \node[draw=none, fill=none, below=1mm of 1] (leaf_1) {\large $1$};
        \node[] (2) at (-1.0,-2.5) {};
        \node[draw=none, fill=none, below=1mm of 2] (leaf_2) {\large $2$};
        \node[] (3) at (-0.5,-2.5) {};
        \node[draw=none, fill=none, below=1mm of 3] (leaf_3) {\large $3$};
        \node[] (4) at (0.0,-2.5) {};
        \node[draw=none, fill=none, below=1mm of 4] (leaf_4) {\large $4$};
        \node[] (5) at (0.5,-2.5) {};
        \node[draw=none, fill=none, below=1mm of 5] (leaf_5) {\large $5$};
        \node[] (6) at (1.0,-2.5) {};
        \node[draw=none, fill=none, below=1mm of 6] (leaf_6) {\large $6$};
        \node[] (7) at (1.5,-2.5) {};
        \node[draw=none, fill=none, below=1mm of 7] (leaf_7) {\large $7$};

        \draw[edge] (0) edge (8);
        \draw[edge] (0) edge (11);
        \draw[edge] (8) edge (9);
        \draw[edge] (8) edge (4);
        \draw[edge] (9) edge (1);
        \draw[edge] (9) edge (10);
        \draw[edge] (10) edge (2);
        \draw[edge] (10) edge (3);
        \draw[edge] (11) edge (12);
        \draw[edge] (11) edge (7);
        \draw[edge] (12) edge (5);
        \draw[edge] (12) edge (6);
    \end{tikzpicture}
    \end{subfigure}%
    \hfill
    \begin{subfigure}[b]{0.49\textwidth}
    \centering
    \begin{tikzpicture}[every node/.style = {draw, circle, fill, inner sep = 0pt, minimum size = 2mm},
square/.style = {regular polygon, regular polygon sides = 4, minimum size = 3 mm}]
    \tikzset {edge/.style = {very thick, shorten >= -0.5 pt}}

        \node[] (0) at (0.0,-0.5) {};
        \node[] (8) at (-0.5,-1.0) {};
        \node[] (16) at (0.5,-1.0) {};
        \node[] (9) at (-0.75,-1.5) {};
        \node[] (11) at (-0.25,-1.5) {};
        \node[] (17) at (1,-2.25) {};
        \node[] (10) at (-1,-2.25) {};
        \node[] (12) at (0,-2.25) {};
        \node[square] (13) at (-0.5,-2.75) {};
        \node[square] (18) at (0.5,-2.75) {};
        \node[] (14) at (-0.5,-3.25) {};
        \node[] (19) at (0.5,-3.25) {};
        \node[square] (15) at (-1.25,-3.5) {};
        \node[square] (20) at (1.25,-3.5) {};

        \node[draw=none, fill=none, left = 5mm of 0] {\large{$N$}};

        \node[] (1) at (-2.0,-4.0) {};
        \node[draw=none, fill=none, below=1mm of 1] (leaf_1) {\large $1$};
        \node[] (2) at (-1.25,-4.0) {};
        \node[draw=none, fill=none, below=1mm of 2] (leaf_2) {\large $2$};
        \node[] (3) at (-0.5,-4.0) {};
        \node[draw=none, fill=none, below=1mm of 3] (leaf_3) {\large $3$};
        \node[] (4) at (0.25,-1.75) {};
        \node[draw=none, fill=none, below=1mm of 4] (leaf_4) {\large $4$};
        \node[] (5) at (0.5,-4.0) {};
        \node[draw=none, fill=none, below=1mm of 5] (leaf_5) {\large $5$};
        \node[] (6) at (1.25,-4.0) {};
        \node[draw=none, fill=none, below=1mm of 6] (leaf_6) {\large $6$};
        \node[] (7) at (2.0,-4.0) {};
        \node[draw=none, fill=none, below=1mm of 7] (leaf_7) {\large $7$};

        \draw[edge, dotted] (0) edge (8);
        \draw[edge, dotted] (0) edge (16);
        \draw[edge, dotted] (8) edge (9);
        \draw[edge, dotted] (8) edge (11);
        \draw[edge, bend right, dotted] (9) edge (1);
        \draw[edge, dotted] (9) edge (10);
        \draw[edge] (10) edge (15);
        \draw[edge, dotted] (10) edge (13);
        \draw[edge, dotted] (15) edge (2);
        \draw[edge, dotted] (13) edge (14);
        \draw[edge, dotted] (14) edge (15);
        \draw[edge, dotted] (14) edge (3);
        \draw[edge, dotted] (11) edge (4);
        \draw[edge] (11) edge (12);
        \draw[edge] (12) edge (13);
        \draw[edge] (12) edge (18);
        \draw[edge, dotted] (18) edge (19);
        \draw[edge] (19) edge (20);
        \draw[edge, dotted] (19) edge (5);
        \draw[edge, dotted] (20) edge (6);
        \draw[edge, dotted] (16) edge (17);
        \draw[edge, bend left, dotted] (16) edge (7);
        \draw[edge, dotted] (17) edge (18);
        \draw[edge, dotted] (17) edge (20);
    \end{tikzpicture}
    \end{subfigure}
    \caption{This is a reproduction of Figure~13 in~\cite{cherry_picking_and_network_containment}. The tree~$T$ is a subnetwork of an orchard network~$N$, and one embedding of $T$ into $N$ is shown (dotted edges).
    There exists no minimal sequence of $N$ that reduces $T$ to a single leaf (\Cref{obs:cor(N)infinite}). For a full explanation, see \Cref{sec:TCOrch}.}
    \label{fig:CPNTreeContainmentCounterExample}
\end{figure}
}

\DeclareMathOperator{\cD}{\mathcal{D}}
\DeclareMathOperator{\LCA}{LCA}

\newtheorem{theorem}{Theorem}[section]
\newtheorem{corollary}[theorem]{Corollary}
\newtheorem{lemma}[theorem]{Lemma}

\newtheorem{observation}[theorem]{Observation}

\theoremstyle{definition}

\title{
How many cherry-picking sequences are needed to reduce all subtrees of a phylogenetic tree?\footnote{The third author was supported by the Waseda University Grant for Special Research Projects (Project number: 2026C-088)}
}
\author[1]{Bálint Kollmann\thanks{B.Kollmann@student.tudelft.nl}}
\author[1]{Yukihiro Murakami\thanks{Y.Murakami@tudelft.nl}}
\author[2]{Takatora Suzuki\thanks{takatora.szk@fuji.waseda.jp}}

\affil[1]{Delft Institute of Applied Mathematics, TU Delft, the Netherlands}
\affil[2]{Department of Applied Mathematics, Faculty of Science and Engineering, Waseda University, Japan}
\date{August 20, 2026}

\begin{document}
\maketitle

\begin{abstract}
Phylogenetic networks are graphs that represent the evolutionary history of species. 
Recently, the class of orchard phylogenetic networks, which can be reduced by so-called \emph{cherry-picking sequences}, has gained attention for its computational and biological aspects.
In this paper, we study a fundamental question on orchards and their cherry-picking sequences by considering the \textsc{CoveringNumber} problem: given an orchard network~$N$, how many cherry-picking sequences are needed to reduce all subnetworks of~$N$?
We initiate this study by considering the problem for trees: \rev{given a phylogenetic tree \(T\) on
    \(n\) leaves, we seek the smallest family of cherry-picking sequences, each of length \(n-1\) and
    reducing \(T\), such that every nonempty subtree of \(T\) is reduced by at least one sequence in the
    family}.
We then show that the covering number can be computed for binary trees recursively using a similar but more fine-grained notion of \emph{survival covering number}.
We also give a recursive formula for the survival covering number of non-binary trees. 
However, computing the covering number for non-binary trees appears to be considerably more challenging.
For this case, we show that the covering number of star trees (whose root is adjacent to all leaves) is equivalent to the so-called {\sc SubsetConnectivity} problem, which we introduce in this paper. 
Finally, we show that if there is no restriction on the sequence length, a single sequence of minimum length~$\binom{n}{2}$ suffices to reduce all subtrees of a tree on~$n$ leaves.
\end{abstract}

\section{Introduction}

Phylogenetic trees have been used to represent the evolutionary history of species since Darwin~\cite{darwin1859origin}. These trees can represent speciation events (one lineage splits into two or more lineages); however, they fail to capture the merging of lineages due to reticulation events such as hybridization and horizontal gene transfer~\cite{Fundamental_work_2013}.
To bridge this gap, a group of mathematicians, computer scientists, and biologists introduced phylogenetic networks (see e.g. ~\cite{huson2010phylogenetic}).

\paragraph{Motivation.}
It is usually the case that biologists can only construct the trees for small stretches of DNA, but not the entire evolutionary history of genomes and species~\cite{gene_trees_vs_species_trees}. Therefore, the evolutionary history of genomes must be inferred by combining the trees of the DNA sequence.
This naturally gives rise to the \textsc{TreeContainment} problem~\cite{cherry_picking_and_network_containment}, which asks whether a phylogenetic tree is displayed by a phylogenetic network.
\rev{Biologically, this is akin to asking if a gene tree with a known evolutionary history is consistent with an inferred network.
Hence solving \textsc{TreeContainment} is an important verification process for network inference.}
\textsc{TreeContainment} is known to be NP-hard in general~\cite{np_hard}, but it can be solved in polynomial-time for certain classes of networks~\cite{algorithm_geneticly_stable, algorithm_nearly_stable, cherry_picking_and_network_containment}.
For example, \textsc{TreeContainment}, among other techniques, can be solved in linear time for \rev{so-called} \emph{tree-child networks} using recently developed cherry-picking techniques \cite{cherry_picking_and_network_containment}. 
Roughly speaking, one iteratively \emph{picks} a lowest subgraph called a \emph{cherry}, until the entire network has been reduced.
Cherry-picking has been studied in the context of ancestral profiles~\cite{erdHos2019class}, edge-coverings~\cite{van2021unifying}, biologically motivated vertex-labellings~\cite{van2022orchard}, and certain linear extensions of topological orderings~\cite{counting_linear_extensions}. 
Cherry-picking techniques scale well with input size when tackling the \textsc{HybridizationNumber} problem (find network with minimum number of reticulations that display all trees) \cite{humphries2013cherry,linz2019attaching,van2022practical}.
\rev{Biologically, this is the question of amalgamating a collection of gene trees in a parsimonious way and obtaining a one `true' evolutionary history of species.}

The linear-time algorithm in \cite{cherry_picking_and_network_containment} exploits the fact that a tree-child network contains a tree on the same leaf set if and only if any \emph{tree-child sequence}, \rev{a special type of cherry-picking sequences}, of the network reduces the tree.
However, if the leaf sets of the tree-child network and the tree are different, then \textsc{TreeContainment} is NP-hard\footnote{by a reduction from general \textsc{TreeContainment}, by adding leaves until the parent network is tree-child. This addition can be done in polynomial-time~\cite{van2023making}.}.
In this case, can we still use cherry-picking?
The reason for writing this paper is to initiate this study by considering the following problem.
\rev{Roughly speaking, a \emph{subtree} of a phylogenetic network is a directed subgraph that forms a rooted tree and whose leaves are all labelled leaves of the phylogenetic network.}

\medskip
\begin{center}
\noindent\fbox{\parbox{0.95\linewidth}{
{\sc CoveringNumber}\\
{\bf Input:} A non-binary phylogenetic network $N$ on $n$ leaves and an integer $k$.\\
{\bf Decide:} Does there exist a set of cherry-picking sequences of size at most $k$ that reduces all subtrees of $N$?}}
\end{center}
\medskip

\rev{For a network~$N$, we write~$c(N)$ as the smallest integer for which~$(N,c(N))$ is a \textsc{YES}-instance of \textsc{CoveringNumber}.}
The first step in tackling \textsc{CoveringNumber} is to check the following: Given a network~$N$ and a subtree~$T$, does there exist a sequence of~$N$ which reduces~$T$? 
For some instances, the answer is no.
See for example the orchard network in \Cref{fig:CPNTreeContainmentCounterExample} and the tree-child network in \Cref{fig:TCNTreeContainmentCounterExample} (see \Cref{sec:TCOrch} for elaboration).
This naturally motivates studying the problem first for trees.



\cpnCounterexample
\tcnCounterexample

The \textsc{CoveringNumber} problem was first introduced for trees by Dee \cite{covering_number_conj}. 
In that thesis, Dee solved the problem for caterpillars (i.e., phylogenetic trees where all leaves are adjacent to a central path) and so-called double caterpillars.
Until now, this problem has remained open for more general trees.

\paragraph{Our contributions.}
We resolve the \textsc{CoveringNumber} problem for binary trees.
We give a recursive formula to compute this number (\Cref{thm:CoveringNumber}) --- which requires a notion of the \emph{survival covering number} and can be computed in linear-time in the input size. Survival covering number is similar to covering number, but it ensures that all subtrees under a vertex are propagated upwards in the tree, so these subtrees can be combined with other subtrees and the combined subtrees can be still reduced.
We show that the survival covering number can also be computed in linear-time for non-binary inputs (\Cref{thm:nonbin_surviving_covering_number}).
The proofs for~\Cref{thm:CoveringNumber} and~\Cref{thm:nonbin_surviving_covering_number} are constructive, so they also suggest a way to construct the sequences which (survivably) reduce all subtrees. 
Interestingly, the covering number problem for a specific input non-binary tree (the star) is equivalent to a combinatorial optimization problem, which we call \textsc{SubsetConnectivity} (\Cref{thm:Equiv}).
Finally, we consider the following question: given a tree on~$n$ leaves, find a shortest cherry-picking sequence which reduces all subtrees.
We show that such a sequence is of length~$\binom{n}{2}$, by selecting all the cherries in an order which, roughly speaking, respects the topological ordering (\Cref{thm:OHSBound}).

\paragraph{Layout of the paper.} 
In this article, we first introduce the notations and definitions in~\Cref{sec:preliminaries} that we use in the paper, such as cherry-picking sequences, covering number, and survival covering number. 
In~\Cref{sec:covering-num-binary}, 
we give a recursive formula to compute the covering number of binary trees (\Cref{thm:CoveringNumber}). 
In~\Cref{sec:surviving-covering-num-non-bin}, we derive a recursive formula for the \textit{survival covering number} of non-binary trees (\Cref{thm:nonbin_surviving_covering_number}).
In~\Cref{sec:NonBinCover}, we prove the equivalence between \textsc{CoveringNumber} of non-binary stars and \textsc{SubsetConnectivity}, and give bounds for the \textsc{SubsetConnectivity} problem.
In~\Cref{sec:Hybrid>cover}, we consider a variant of \textsc{CoveringNumber} where we only allow for one sequence of any length.
In \Cref{sec:Discussion}, we give suggestions on future research.
In \Cref{sec:TCOrch}, we consider \textsc{CoveringNumber} for tree-child and orchard networks.

\section{Preliminaries}\label{sec:preliminaries} 

In this section, we introduce definitions and notation used in the paper.
For an integer~$n\in\mathbb{N}$, we write $[n]:=\{1,\ldots, n\}$.
Let~$X$ be a non-empty set. 
By a \emph{phylogenetic network} on~$X$, we mean an acyclic directed graph $(V,E)$ with 
\begin{itemize}
    \item a unique \emph{root} of indegree-$0$ and \rev{outdegree at least $2$};
    \item \emph{tree vertices} of indegree-$1$ and outdegree at least $2$;
    \item \emph{reticulations} of indegree at least~$2$ and outdegree-$1$;
    \item \emph{leaves} of indegree-1 labelled bijectively by~$X$.
\end{itemize}
For brevity, we shall write network to mean a phylogenetic network. 
We say that a vertex is \emph{binary} if it has outdegree-$2$ or indegree-$2$.
We say that a vertex is \emph{non-binary} if it has outdegree at least~$3$ or indegree at least~$3$.
We say that a network is \emph{binary} if the root and all tree vertices have outdegree-2 and all reticulations have indegree-2 (i.e., it contains no non-binary vertices).
We say that a network is \emph{non-binary}, otherwise, to mean that it is `not necessarily binary'.
A network without reticulations is called a \emph{phylogenetic tree}. 
See \Cref{fig:PhyloTree} for an example of a phylogenetic tree on~$\{1,2,3,4,5\}$, together with illustrations of other definitions in this section. 
Again, we call these objects trees (without the phylogenetic), if there is no ambiguity.
In particular, in \Cref{sec:NonBinCover}, we deal with phylogenetic trees and undirected trees simultaneously; there, the distinction will be made clear between the two objects.
Given an edge~$uv$ in a network, we say that~$u$ is the \emph{parent} of~$v$ and that~$v$ is the \emph{child} of~$u$.
For two distinct vertices $u, v$ of a network, we say that $v$ is \emph{below} $u$ if there exists a path from $u$ to $v$; we say \emph{strictly below} if this path is non-trivial.
In this case, we also say that~$u$ is \emph{above} and \emph{strictly above}~$v$, respectively.
\yuki{Two vertices are \emph{incomparable} if neither vertex is above the other.}

\begin{figure}
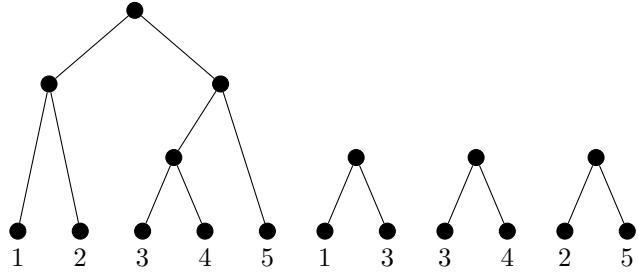

\centering
\treePreliminaries{1}
\cherry{1}{1}{3}
\cherry{1}{3}{4}
\cherry{1}{2}{5}
\caption{
A phylogenetic tree~$T$ on~$\{1,2,3,4,5\}$, and three of its subtrees $T|_{\{1,3\}}$, $T|_{\{3,4\}}$ and $T|_{\{2,5\}}$.
The subtree~$T|_{\{3,4,5\}}$ is a pendant subtree of~$T$ rooted at the parent of~$5$.
The subtree~$T|_{\{1,3\}}$ is not a pendant subtree.
\rev{The subtrees $T|_{\{3,4\}}$ and $T|_{\{2,5\}}$ are reducible with the same sequence $(1,2)(3,4)(4,5)(2,5)$, but they are pairwise survivably irreducible, because they share no leaves. In this case $T|_{\{2,5\}}$ survives with the surviving leaf $5$, but $T|_{\{3,4\}}$ is not a surviving subtree.} The subtrees $T|_{\{1,3\}}$ and $T|_{\{2,5\}}$ are pairwise irreducible (and hence pairwise survivably irreducible).
}
\label{fig:PhyloTree}
\end{figure}

For the rest of the section, let~$T=(V,E)$ be a tree on~$X$.
Let~$v\in V$.
We write $L[v]$ to denote all leaves below a vertex $v$; let $s(v) = |L[v]|$.
Let $L[T]$ denote the set of all leaves of tree $T$.
Let $X'$ be a subset of $X$. 
The tree~$T|_{X'}$ is obtained from $T$ by removing all leaves from $T$ whose labels are in $X \setminus X'$, and repeating the following two procedures as long as it is applicable.
\begin{itemize}
    \item Remove a vertex with indegree-$1$ and outdegree-$1$ (\emph{suppress}), and add an edge between its parent and child.
    \item Remove leaves that are unlabeled.
\end{itemize}
We say that the \emph{subtree} $T|_{X'}$ is \emph{induced} by $X'$ (see \Cref{fig:PhyloTree}). 
\taka{We denote the set of all subtrees of a tree $T$ by $\mathcal{T}_T$.}
Consider two sets $P,Q \subseteq \mathcal{T}_T$.
We write $P \times Q = \{T|_{L[T_1] \cup L[T_2]}: T_1\in P, T_2\in Q\}$ to be the set of subtrees obtained by joining the elements of~$P$ and~$Q$.
We call $T[v] := T|_{L[v]}$ the \emph{subtree of $T$ rooted at $v$}.
Such subtrees are also called \emph{pendant subtrees} of $T$ (See \Cref{fig:PhyloTree}).

Let~$V' \subseteq V$ be a set of vertices in~$T$.
The \emph{lowest common ancestor} of $V'$ is the lowest vertex which is above every vertex in $V'$. We denote the lowest common ancestor of a set of vertices $V'$ in $T$ by $\LCA_T(V')$ and denote the lowest common ancestor of vertices $a$ and $b$ in tree $T$ by $\LCA_T(a,b)$.

Let~$T_1,T_2$ be trees on~$X$.
We say~$T_1=(V_1,E_1)$ and~$T_2=(V_2,E_2)$ are \emph{isomorphic}, denoted by~$T_1\cong T_2$, if there is a leaf-label-preserving bijection~$f:V_1\rightarrow V_2$ where~$uv\in E_1$ if and only if~$f(u)f(v)\in E_2$.




\paragraph{Cherry-picking sequences.} 

We say that two leaves~$x,y$ form a \emph{cherry}~$(x,y)$ if they have the same parent.
Note that a tree always has at least one cherry.
We \emph{reduce} a cherry~$(x,y)$ in~$T$ by deleting the leaf~$x$ and suppressing any degree-2 vertices in~$T$.
Note that suppression may not occur when considering non-binary networks, e.g., when a vertex is a parent of three leaves~$x,y,z$.
The resulting tree is denoted by~$T(x,y)$.
If~$(x,y)$ is not a cherry of~$T$, then we say that~$(x,y)$ \emph{does not affect}~$T$ and we have~$T(x,y)\cong T$.

Let~$S$ be a sequence of ordered pairs on distinct elements from~$X$.
We say that~$S$ is a \emph{cherry-picking sequence} if the second coordinate of every pair appears either as the first coordinate of a later pair or as the second coordinate of the final pair.
A \emph{subsequence} of $S$ is a sequence obtained from $S$ by deleting zero or more pairs, keeping the remaining pairs in their original order. 
In the rest of the paper, we refer to cherry-picking sequences as \emph{sequences}.

Then~$TS$ is the tree obtained by iteratively applying the elements of~$S$ to~$T$. 
\taka{We say that $S$ \emph{reduces} $T$ if $TS$ is a tree with a single leaf.}
We denote the length of $S$ by $|S|$.
Observe that the shortest cherry-picking sequences that reduce a tree~$T$ contain an element as a first coordinate at most once, and it has length~$|L[T]|-1$ (for more general sequences, see \Cref{sec:TCOrch}).
If $S$ is a cherry-picking sequence of length $|L[T]| - 1$ that reduces $T$, we say that $S$ is a \emph{sequence of} $T$.

\paragraph{Covering Number.}
A set of sequences $\mathcal{S}$ is a \emph{covering set} of a set $\mathcal{T}$ of subtrees of $T$ if each element of $\mathcal{S}$ is a sequence of $T$, and every subtree in $\mathcal{T}$ is reduced by at least one sequence in $\mathcal{S}$.
In this case, we also say that $\mathcal{S}$ \emph{covers} $\mathcal{T}$.
The \emph{covering number} of $T$, written~$c(T)$, is the size of the smallest covering set of $\mathcal{T}_T$. The smallest covering set is called \emph{optimal}.
We write $c_T(v)$ to denote the covering number of the subtree of $T$ rooted at $v$, and simply write~$c(v)$ if there is no ambiguity in the reference tree. 
In particular, $c(T) = c(\rho)$ where $\rho$ is the root of $T$.

We call two subtrees of $T$ \emph{pairwise irreducible}, if there is no sequence of $T$ that reduces both of them. 
We call a set of subtrees of~$T$ a \emph{pairwise irreducible subtree set} if all pairs of trees from the set are pairwise irreducible. 
By definition, a set that contains only one subtree or a set that is empty are also considered to be pairwise irreducible subtree sets.

\medskip
\begin{center}
\noindent\fbox{\parbox{0.95\linewidth}{
{\sc CoveringNumber}\\
{\bf Input:} A phylogenetic tree $T$ on $n$ leaves and an integer $k$.\\
{\bf Decide:} Does there exist a set of cherry-picking sequences of~$T$ of size at most $k$ that covers $T$?}}
\end{center}

\paragraph{Survival Covering Number.}
Let $S$ be a sequence of $T$ and let~$T'$ be a pendant subtree of~$T$. 
Let~$(x_k,y_k)$ be the last element in~$S$ which reduces a cherry in~$T'$.
We call~$y_k$ the \emph{surviving leaf} of~$S$ with respect to~$T'$.
Let $T''$ be a subtree of $T'$ that is reduced by $S$. 
We say that $T''$ is a \emph{surviving subtree} with respect to $S$ and~$T'$, if the remaining leaf in $T''S$ is~$y_k$.
In this case, we say that~$T''$ is \emph{survivably reducible} with respect to~$S$ and~$T'$, or $S$ \emph{survivably reduces} $T''$ with respect to~$T'$.


A set of sequences $\mathcal{S}$ is a \emph{survival covering set} of a set $\mathcal{T}$ of subtrees of $T$ if each element of $\mathcal{S}$ is a sequence of $T$, and every subtree in $\mathcal{T}$ is survivably reducible with respect to at least one sequence in $\mathcal{S}$ and $T$.
In this case, we also say that $\mathcal{S}$ \emph{survivably covers} $\mathcal{T}$.
The \emph{survival covering number} of $T$, written $c'(T)$, is the size of the smallest survival covering set of $\mathcal{T}_T$. The smallest survival covering set is called \emph{optimal}.
We write $c'_T(v)$ to denote the survival covering number of the subtree of $T$ rooted at $v$, and simply write $c'(v)$ if there is no ambiguity in the reference tree.
In particular, $c'(T) = c'(\rho)$ where $\rho$ is the root of $T$.

\balint{We call two subtrees $T_1$ and $T_2$ of $T$ \emph{pairwise survivably irreducible}, if there is no sequence $S$ of $T$ such that $T_1$ and $T_2$ are both surviving subtrees with respect to $S$ and $T$. 
We call a set $\psi$ of subtrees of $T$ a \emph{pairwise survivably irreducible subtree set}, if all pairs of subtrees from the set are pairwise survivably irreducible.
\rev{This means that at most one element of $\psi$ can be survivably reduced for any sequence $S$ of $T$.}
By definition, a set that contains only one subtree or a set that is empty are also considered to be pairwise survivably irreducible subtree sets.
Let $v \in V$.
We write $\psi_v$ to be a pairwise survivably irreducible subtree set of maximum size of the subtrees in $\mathcal{T}_{T[v]}$.
}

Note that if two subtrees~$T_1,T_2$ of~$T$ are pairwise irreducible, then they are also pairwise survivably irreducible.




\section{Covering number of a binary tree}\label{sec:covering-num-binary}

In this section, we give a recursive formula to calculate the covering number of a binary tree.
The main goal of this section is to prove \Cref{thm:CoveringNumber}. 

\begin{restatable}{theorem}{thmCoveringNumber}\label{thm:CoveringNumber}
Let $T$ be a binary tree on $X$ rooted at $\rho$ with children $a$ and $b$, such that $s(a) \ge s(b)$. Then 
\[ c(\rho) = \begin{cases} 
      c'(a)c'(b) & \text{if } s(b) \ge 2 \\
      c'(a) & \text{if } s(b) = 1 \rev{\text{ and } s(a) > 1} \\
      1 & \text{if } s(a) = 1. \\
   \end{cases}
\]
\end{restatable}

To achieve this, we first show characterizations for constructing pairwise survivably irreducible subtree sets in \Cref{subsec:ConstPISubtrees} (\Cref{lem:non-reducible-intersect,lem:non-reducible-concat}). 
In \Cref{subsec:SCNBounds}, we use these pairwise survivably irreducible subtree sets to prove a lower bound for the maximum size of such a set (which is also a lower bound for the survival covering number) in \Cref{lem:lower_bound}. We then show an upper bound for the survival covering number and prove that the lower bound and the upper bound are equal. Finally, we prove the recursive formula for the 
covering number. 

While \Cref{thm:CoveringNumber} is a result on binary trees, we prove some of the lemmas for non-binary trees to be as general as possible.
These more general lemmas will be useful in \Cref{sec:surviving-covering-num-non-bin}.
For the rest of the section, let~$T=(V,E)$ be a tree on~$X$, where~$|X|=n$.


\subsection{Constructing Pairwise Survivably Irreducible Subtrees}\label{subsec:ConstPISubtrees}

We first show that subtrees with no common leaves are not survivably reducible by the same sequence.

\begin{lemma}\label{lem:non-reducible-intersect}
    Let~$T$ be a non-binary tree on $X$.
    Let~$P,Q\subseteq X$ such that~$P\cap Q=\emptyset$.
    Then $T|_P$ and $T|_Q$ are pairwise survivably irreducible.
\end{lemma}

\begin{proof}
    By definition, every sequence with respect to~$T$ has exactly one surviving leaf.
    Since $T|_P$ and $T|_Q$ have no common leaf, the claim follows.
\end{proof}

We now consider two subtrees which may contain common leaves.
Roughly speaking, we show that if they contain subtrees which are pairwise irreducible, then the two original subtrees are pairwise irreducible (\Cref{lem:non-reducible-concat}).
We prove a series of helper lemmas to prove this result.

\begin{lemma}\label{lem:Subsequence}
Let~$T$ be a non-binary tree on $X$, and let~\rev{$A\subseteq X$}, where~$|A|\ge 2$.
Let $S$ be a sequence \rev{of} $T$.
Then $S$ reduces $T|_A$
if and only if $S$ has a subsequence
$S'$ of size $|A|-1$ that contains only pairs of elements of $A$.
\end{lemma}

\begin{proof}
    Suppose first that $S$ reduces $T|_A$.
    Consider the shortest subsequence $S'$ of $S$ that reduces $T|_A$. 
    Each time we reduce a cherry, we remove exactly one leaf. 
    As $T|_AS'$ is a tree on one leaf,~$S'$ removes $|A| - 1$ leaves from $T|_A$, and so $|S'| = |A| - 1$.
    We now show that $S'$ contains only the elements of $A$.
    Note that if there is an element $(x,y)$ in $S'$ such that $|\{x,y\} \cap A| \le 1$, then when applying $S'$ to $T|_A$, the element $(x,y)$ does not affect $T|_A$.  
    Let~$S'\setminus (x,y)$ denote the sequence~$S'$ with \rev{~$(x,y)$ removed}.
    But then $S' \setminus (x,y)$ is a shorter subsequence of $S$ that also reduces $T|_A$, which contradicts our choice of $S'$.
    Therefore,~$S'$ contains only pairs of elements of $A$.
    Since $|S'| = |A|-1$ and $S'$ only contains pairs of elements of $A$, the claim follows. 
    \medskip

    For the other direction, suppose now that $S$ has a subsequence $S'$ of size $|A|-1$ that contains only pairs of elements of $A$.
    We prove by induction on the size of $A$, \yuki{that~$S$ reduces~$T|_A$.}
    For the base case, take the set $A = \{x,y\}$.
    By assumption, either $(x,y)$ or $(y,x)$ is an element of $S$. Both of these elements reduce the single cherry $T|_A$.
    
    For the induction step, suppose that the claim holds for all sets \rev{$A\subseteq X$} with $2 \le|A| \le k$, for $k \ge 2$. We show that the claim is true for any set \rev{$A\subseteq X$} of size $k+1$.
    By assumption, let $S' = (x_1,y_1)(x_2,y_2)\ldots(x_k,y_k)$ \yuki{be a subsequence of~$S$ of size $|A|-1 = k$ containing only pairs of elements of~$A$.}
    \yuki{We now prove the following facts which will be used in the rest of the proof}.
    
    \paragraph{Fact 1:} There is no sequence $S''$ which contains only pairs of elements of $A$ and $|S''| > |S'|$. 
    Indeed, if one of the elements appears as a first coordinate in $S$, it cannot appear later anywhere in $S$, so the maximum possible length of a sequence that contains only pairs of elements of $A$ is $|A|-1$. 
    
    \paragraph{Fact 2:} There is no pair $(a,b) \in S$ such that $a \in A \setminus \{y_k\}$ and $b \notin A$. Otherwise, there would be at most $|A|-2$ elements that could appear as a first coordinate in $S'$ and this would mean that $|S'| \le |A|-2$.
    \medskip
    
    We claim $(x_1, y_1)$ is a cherry in $T|_A$.
    Suppose $(x_1, y_1)$ is not a cherry in $T|_A$. Then there must exist a leaf $z\in A$ such that $\LCA_{T|_A}(x_1,z)$ or $\LCA_{T|_A}(y_1,z)$ is strictly below $\LCA_{T|_A}(x_1,y_1)$. And because of this and the fact that $T|_A$ is a subtree of $T$, $\LCA_T(x_1,z)$ or $\LCA_T(y_1,z)$ is strictly below $\LCA_T(x_1,y_1)$.
    And since $S$ reduces $T$, $z$ must appear as a first coordinate before the pair $(x_1,y_1)$, because otherwise $(x_1,y_1)$ could not become a cherry in $T$. 
    \yuki{Let~$(z,w)$ denote this element.}
    Note that $z \ne y_k$, \yuki{as~$y_k$ appears in the final element of~$S'$.} 
    By \textbf{Fact 2}, $w \in A$.
    So $(z,w)S'$ is also a sequence that contains only pairs of elements of $A$ and $|(z,w)S'| = |A|$. But this is a contradiction as $|(z,w)S'| \le |A|-1$  by \textbf{Fact 1}.
    Therefore, $(x_1,y_1)$ must be a cherry in $T|_A$.
    
    Let $S'' = S' \setminus (x_1,y_1)$.
    By the induction hypothesis, $S''$ reduces $T|_{A \setminus\{x_1\}}$. And since $(x_1,y_1)$ is a cherry in $T|_A$, $(x_1,y_1)S'' = S'$ reduces $T|_A$.
    \yuki{It follows that~$S$ also reduces~$T|_A$.}
    
    So both directions of the claim are true.
\end{proof}


\begin{lemma}\label{lem:reduce-subtrees}
    Let~$\hat{T}$ be a non-binary tree on $X$ and let~$T$ be a subtree of~$\hat{T}$. 
    Let~$S$ be a sequence of~$\hat{T}$ which reduces~$T$.
    Then $S$ reduces all pendant subtrees of $T$.
\end{lemma}

\begin{proof}
    

    We prove the claim
    by induction on the number of leaves of $T$.
    For the base case, take $|L[T]| = 2$. In particular $T$ is a cherry $(a,b)$. The only pendant subtrees of $T$ are itself, which is reduced by the sequence $(a,b)$, \yuki{and the singleton trees on leaves~$a$ and~$b$, which are also reduced by~$(a,b)$ (in particular, they are already reduced).}

    Suppose the claim holds for every tree on at most $|L[T]| = k$ leaves. Consider a tree $T$ on $k+1$ leaves.
    \rev{Let $\hat{S}$ be any sequence of $\hat{T}$ that reduces~$T$, and let the first element of $\hat{S}$ that affects~$T$ be $(x,y)$. 
    Let $S$ be the sequence after this occurrence of~$(x,y)$, in other words, $\hat{S} = S'(x,y)S$.
    Note that $(x,y)$ is a cherry of $T$, because $(x,y)$ is the first element of $\hat{S}$ which affects $T$. So observe that $S$ is a sequence that reduces $T(x,y)$.}
    Let~$v$ be a vertex in~$T$.
    \paragraph{Case 1: $x\notin L[v]$.}
    In this case, $T|_{L[v]}$ is a pendant subtree of $T(x,y)$. So \rev{$S'(x,y)S = \hat{S}$} reduces $T|_{L[v]}$ by induction hypothesis, \rev{since~$S'$ does not affect~$T$}.
    \paragraph{Case 2: $x\in L[v]$.}
    In this case, $T|_{L[v]}(x,y)$ is a pendant subtree of $T(x,y)$. 
    By induction hypothesis,~$T|_{L[v]}(x,y)$ is reduced by~$S$, which means~$T|_{L[v]}$ is reduced by \rev{$S'(x,y)S = \hat{S}$, since~$S'$ does not affect~$T$}.
\end{proof}

\begin{lemma}\label{lem:subtree-must-survive-to-reduce-whole-tree}
    Let $T$ be a non-binary tree on $X$, let $v \in V$ and $A \subseteq L[v]$, and let $B \subseteq X \setminus L[v]$ where $B \neq \emptyset$. 
    Let $S$ be a sequence of $T$ that reduces $T|_{A\cup B}$. 
    Then $T|_A$ is survivably reducible with respect to $S$ and the subtree $T[v]$.
\end{lemma}

\begin{proof}
    Note that $T|_A$ is a pendant subtree of $T|_{A\cup B}$, and $S$ reduces $T|_{A\cup B}$ by assumption. 
    So by \Cref{lem:reduce-subtrees}, $S$ also reduces $T|_A$. 
    Note that there must be an element $(a,b)$ or $(b,a)$ in $S$ such that $a \in A$ and $b \in B$, otherwise $S$ would not reduce $T|_{A\cup B}$. 
    
    Suppose first that $(a,b)$ is in $S$.
    \yuki{We wish to show that no element of~$L[v]$ appears in any pair of~$S$ after~$(a,b)$.}
    Suppose there is an element $(x,y)$ or~$(y,x)$ of $S$, such that $x \in L[v], y\in X$, and it appears later in $S$ than $(a,b)$.
    We consider the~$(x,y)$ and the~$(y,x)$ cases simultaneously.
    \yuki{Note that~$x\ne a$ since each leaf can appear as a first coordinate at most once.}
    When we apply $(a,b)$ \yuki{to $T$}, the leaves $a$ and $x$ (\yuki{and possibly~$y$}) are still under $v$. 
    Therefore, the $(a,b)$ in the sequence~$S$ would not affect $T$, which contradicts the assumption that $S$ is a sequence of $T$. 
    So no element of~$L[v]$ appears in any pair of~$S$ after~$(a,b)$. 
    
    By an analogous argument used above, if~$(b,a)$ is an element of $S$ instead of $(a,b)$, then for all elements~$(x,y)$ of~$S$ after~$(b,a)$, if an element of~$L[v]$ appears, then it must be the leaf~$a$.
    
    So in both cases~$(a,b)$ and~$(b,a)$, it follows that~$a$ is the surviving leaf in~$T[v]$ with respect to~$S$.
    By \Cref{lem:reduce-subtrees},~$T[v]$ is reduced by~$S$ as it is a pendant subtree of $T$. 
    Therefore, $a$ must be the surviving leaf of $T[v]$ with respect to $S$ and $T[v]$. 
    Because $a$ is the surviving leaf in $T[v]$ with respect to $S$, $a \in A$, and $S$ reduces $T|_A$, it follows that $T|_A$ is survivably reducible with respect to $S$ and $T[v]$.
\end{proof}

\begin{lemma}\label{lem:non-reducible-concat}
    Let $T$ be a non-binary tree on $X$. Let $v \in V$ and let $A,B \subseteq L[v]$, where $T|_A$ and $T|_B$ are pairwise survivably irreducible. 
    Let $P,Q \subseteq X \setminus L[v]$ be non-empty. 
    Then $T|_{A \cup P}$ and $T|_{B \cup Q}$ are pairwise irreducible (See \Cref{fig:concat-lemma-figure}).
\end{lemma}

\begin{proof}
    If a sequence $S$ of $T$ reduces $T|_{A\cup P}$, then $T|_A$ must be survivably reducible with respect to $S$ and $T[v]$ by \Cref{lem:subtree-must-survive-to-reduce-whole-tree}.
    Similarly, if $S$ reduces $T|_{B \cup Q}$, then $T|_B$ must be survivably reducible with respect to $S$ and $T[v]$ by \Cref{lem:subtree-must-survive-to-reduce-whole-tree}.
    \yuki{But by assumption,~$T|_A$ and~$T|_B$ are pairwise survivably irreducible.}
    \yuki{So at most one of} $T|_A$ and $T|_B$ is survivably reducible with respect to any sequence of $T$ and subtree $T[v]$, and the claim follows.
\end{proof}

\begin{figure}[ht]
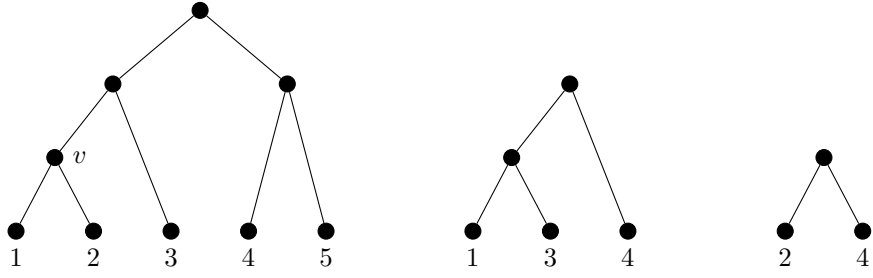

    \centering
    \begin{minipage}[b]{0.4\textwidth}
        \centering
        \treeMultiplicationLemma{0.8}
    \end{minipage}
    \begin{minipage}[b]{0.3\textwidth}
        \centering
        \treeMultiplicationA{0.8}
    \end{minipage}
    \begin{minipage}[b]{0.2\textwidth}
        \centering
        \treeMultiplicationB{0.8}
    \end{minipage}
    \caption{An example of a tree $T$ (on the left) and two of its pairwise irreducible subtrees ($T|_{\{1,3,4\}}$ and $T|_{\{2,4\}}$) in \Cref{lem:non-reducible-concat}. Let $A = \{1\}$, $B = \{2\}$, $P = \{3,4\}$ and $Q = \{4\}$. Then $T|_{A \cup P}$ and $T|_{B \cup Q}$ are pairwise irreducible with respect to any sequence of $T$, because either $T|_A$ or $T|_B$ can survive with respect to the subtree $T[v]$ and any sequence of $T$, but not both.}
    \label{fig:concat-lemma-figure}
\end{figure}

\subsection{Bounds on the survival covering number}\label{subsec:SCNBounds}
We start with proving a lower bound on the survival covering number (\Cref{lem:lower_bound}). We do this by constructing a set of pairwise survivably irreducible subtrees using \Cref{lem:non-reducible-intersect,lem:non-reducible-concat}.
Then we prove an upper bound on the survival covering number (\Cref{lem:upper_bound}) by constructing a set of sequences that survivably reduce every subtree of a tree. Finally, we show that the lower and upper bounds are equal (\Cref{thm:surviving_covering_number}, which leads to a recursive formula for the survival covering number (\Cref{cor:surviving_covering_number}).

\yuki{Recall that for a vertex~$v\in V$, we write~$s(v) = |L[v]|$ to denote the number of leaves below~$v$.
We write $\psi_v$ to be a pairwise \balint{survivably} irreducible subtree set of maximum size of the subtrees of $T[v]$. 
If~$S$ survivably reduces a tree in~$\psi_v$ with respect to $T[v]$, then it does not survivably reduce the other trees in~$\psi_v$ with respect to $T[v]$. 
}

\begin{observation}\label{obs:anc-surv-red}
    \balint{Let $T$ be a non-binary tree and let~$u,v$ be vertices where~$u$ is above~$v$. 
    No two elements of $\psi_v$ can be survivably reduced with respect to any sequence of $T[u]$ and \rev{the pendant subtree} $T[u]$.}
\end{observation}
\begin{proof}
    Let $S$ be a sequence of $T[u]$ and $S'$ be the subsequence of $S$ consisting of every element that affect $T[v]$. Then $S'$ cannot survivably reduce two distinct elements of $\psi_v$. Note that the element of $S$ not in $S'$ do not affect any \rev{cherry} of $T[v]$, and the claim follows.
\end{proof}

\begin{lemma}\label{lem:lower_bound}
    Let $T$ be a binary tree on $X$ rooted at $\rho$ with children $a$ and $b$, such that $s(a) \ge s(b)$. Then 
\[ c'(\rho) \ge |\psi_\rho| \ge \begin{cases} 
      |\psi_a||\psi_b| & \text{if } s(b) \ge 2 \\
      |\psi_a| + 1 & \text{if } s(b) = 1. \\
   \end{cases}
\]
\end{lemma}

\begin{proof}
    Note that $c'(\rho) \ge |\psi_\rho|$, because by definition, no two subtrees can be survivably reduced with respect to a single sequence in $\psi_\rho$ and $T$, so we need a distinct sequence for each of the subtrees in $\psi_\rho$.

    We claim that no two subtrees in $\psi_a \cup \psi_b$ can be survivably reduced with respect to the same sequence and $T$. To prove this, take pairs of elements from $\psi_a \cup \psi_b$ and consider two cases:
    \begin{itemize}
        \item If both elements in the pair come from $\psi_a$ or both of the elements come from $\psi_b$, then they cannot be survivably reduced with respect to the same sequence and $T$ by definition of $\psi_a$ and $\psi_b$, and \Cref{obs:anc-surv-red}.
        \item If the two elements of the pair come from different sets, one from $\psi_a$ and from $\psi_b$, then they have no common leaves, so they are pairwise survivably irreducible by \Cref{lem:non-reducible-intersect}.
    \end{itemize}
    Since no two elements in $\psi_a \cup \psi_b$ can be survivably reduced by the same sequence, $|\psi_\rho| \ge |\psi_a \cup \psi_b| = |\psi_a| + |\psi_b|$.

    We now show that no two elements of $\psi_a \times \psi_b$ can be survivably reduced with respect to the same sequence and $T$.
    Take two non-isomorphic trees $P$ and $Q$ from $\psi_a \times \psi_b$. Let $P_a = T|_{L[a] \cap L[P]}$. Define $P_b$, $Q_a$, and $Q_b$ analogously.
    Since $P \not\cong Q$, we have either $P_a \not\cong Q_a$, or $P_b \not\cong Q_b$. 
    \yuki{Without loss of generality, suppose~$P_a\not\cong Q_a$.}
    \yuki{Since~$P_a,Q_a\in \psi_a$, they are pairwise \balint{survivably} irreducible.
    By \Cref{lem:non-reducible-concat},~$P$ and~$Q$ are also pairwise survivably irreducible, and hence they are not survivably reducible with respect to any sequence of~$T$ and~$T$.}
    \yuki{This holds for all pairs of elements in $\psi_a \times \psi_b$, and so}
    $|\psi_\rho| \ge |\psi_a \times \psi_b| = |\psi_a||\psi_b|$.

    Therefore, $|\psi_\rho| \ge \max(|\psi_a|+|\psi_b|, |\psi_a||\psi_b|)$. Recall that $s(a) \ge s(b)$. We have $|\psi_a| = 1$ if and only if $s(a) = 1$. The same holds for $\psi_b$ and $b$. 
    \yuki{As} $|\psi_a| + |\psi_b| > |\psi_a||\psi_b|$ if and only if $s(b) = 1$,
    we have \[ c'(\rho) \ge |\psi_\rho| \ge \begin{cases} 
      |\psi_a||\psi_b| & \text{if } s(b) \ge 2\\
      |\psi_a| + 1 & \text{if } s(b) = 1.
   \end{cases}
\]
\end{proof}

\yuki{We now give an upper bound on the survival covering number.}

\begin{lemma}\label{lem:upper_bound}
    Let $T$ be a binary tree on $X$ rooted at $\rho$ with children $a$ and $b$, where $s(a) \ge s(b)$. Then 
\[ c'(\rho) \le \begin{cases} 
      c'(a)c'(b) & \text{if } s(b) \ge 2 \\
      c'(a) + 1 & \text{if } s(b) = 1 \rev{\text{ and } s(a) > 1} \\
      2 & \text{if } s(a) = 1. \\
   \end{cases}
\]
\end{lemma}

\begin{proof}
Recall that $\mathcal{T}_T$ denotes the set of all subtrees of a tree $T$.
If~$c'(a)\ge 2$, we let~$\{S_1, S_2, ..., $\\$S_{c'(a)}\}$ be an optimal survival covering set for $\mathcal{T}_{T[a]}$, and let~$a_i$ be the surviving leaf of~$S_i$ with respect to~$T[a]$.
Similarly, if~$c'(b)\ge 2$, we let~$\{R_1, R_2, ..., R_{c'(b)}\}$ be an optimal survival covering set for $\mathcal{T}_{T[b]}$, and let~$b_j$ be the surviving leaf of~$R_j$ with respect to~$T[b]$.
We split into three cases, when $s(b) \ge 2$, $s(b) = 1$ and $s(a) = 1$.

\paragraph{Case 1: $s(b) \ge 2$.}
We refer the reader to \Cref{fig:Sur_upper_bound} for an illustration of the following construction.
Consider the following triplets $(S_1, a_1, P_1),\ldots,(S_{c'(a)}, a_{c'(a)},P_{c'(a)})$ where for each $i \in [c'(a)]$, 
$P_i$ is the set of subtrees that $S_i$ survivably reduces with respect to~$T[a]$.
Define the triplets $(R_1, b_1, Q_1), ..., (R_{c'(b)}, b_{c'(b)}, Q_{c'(b)})$ analogously.
We construct a new set of $c'(a)c'(b)$ sequences of~$T$ using these triplets and show that the constructed sequences form a survival covering set of $\mathcal{T}_T$. 
To do so, we first give an ordering on the leaves below $a$ and $b$. 
Let $v$ be a vertex in $T$. Then we define an ordering on $L[v]$ with respect to $v$ as an arbitrary bijection
\[Ord_v: L[v] \rightarrow \{1, 2, \ldots, s(v)\}.\]
Define a matrix $D$ with $c'(a)$ rows and $c'(b)$ columns, whose $(i,j)$-th element is
\[D_{i,j} = \begin{cases} 
      S_iR_j(a_i,b_j) & \text{if } Ord_a(a_i) = Ord_b(b_j) \\
      S_iR_j(b_j,a_i) & \text{if } Ord_a(a_i) \ne Ord_b(b_j)\\
   \end{cases}
\]
It remains to show that every subtree of $T$ is survivably reduced with respect to one of these $c'(a)c'(b)$ sequences and $T$.
Let~$T'$ be a subtree of~$T$.
There are three subcases to consider: $T'$ has leaves only in~$L[a]$, only in~$L[b]$, or in both~$L[a]$ and~$L[b]$.

\paragraph{Subcase 1: $T'$ has leaves only in~$L[a]$.}
By \yuki{construction,}
there exists an $i\in [c'(a)]$ where $T'$ is survivably reducible with respect to $S_i$ and $T[a]$, with~$a_i$ as the surviving leaf. 
By definition of $Ord_a$ and $Ord_b$, and by assumption that $s(b) \ge 2$, there exists $j\in [\rev{c'(b)}]$ \rev{with $R_j$ having surviving leaf $b_j$} such that $Ord_a(a_i) \neq Ord_b(b_j)$.

\yuki{By construction of the matrix~$D$, we have}
$D_{i,j} = S_iR_j(b_j,a_i)$.
Combining all the arguments above, $T'S_i = T'|_{\{a_i\}}$ is the singleton tree on $a_i$.
\yuki{No element affects singleton trees}, so~$T'S_iR_j(b_j,a_i) = T'|_{\{a_i\}}$.
As~$a_i$ is the second coordinate of the last element of $S_iR_j(b_j,a_i)$, it follows that $T'$ is survivably reducible with respect to $D_{i,j}$ and $T$.

\paragraph{Subcase 2: $T'$ has leaves only in~$L[b]$.}
By \yuki{construction},
there exists a $j\in[c'(b)]$, where $T'$ is survivably reduced with respect to $R_j$ and $T[b]$ \yuki{with~$b_j$ as the surviving leaf}.
Furthermore, $S_i$ does not reduce any cherry in $T'$ for all $i$. By definition of $Ord_a$ and $Ord_b$ and the assumption that $s(a) \ge s(b)$, there exists a $k\in[s(b)]$ such that $Ord_a(a_m) = k$ and $Ord_b(b_j) = k$ for some $m\in[c'(a)]$. 

\yuki{By construction of the matrix~$D$, we have}
$D_{m,j} = S_mR_j(a_m,b_j)$.
Combining all arguments above, we have $T'S_m = T'$; observe that $T'S_mR_j = T'|_{\{b_j\}}$ is the singleton tree on $b_j$ and $T'S_mR_j(a_m,b_j) = T'|_{\{b_j\}}$, so $b_j$ is a surviving leaf with respect to $D_{m,j}$ and~$T$. 
Therefore, $T'$ is survivably reducible with respect to $D_{m,j}$ and $T$.

\paragraph{Subcase 3: $T'$ has leaves in both~$L[a]$ and~$L[b]$.}
Let $T_1 = T|_{L[a]\cap L[T']}$ and
$T_2 = T|_{L[b]\cap L[T']}$.
\yuki{Since~$T_1,T_2$ are subtrees of~$T[a],T[b]$, respectively, there exist~$S_i$ and~$R_j$ such that~$S_i$ survivably reduces~$T_1$ with respect to~$T[a]$, and~$R_j$ survivably reduces~$T_2$ with respect to~$T[b]$.}
Note that $S_i$ only reduces cherries that have both leaves in $T_1$, and~$R_j$ only reduces cherries that have both leaves in $T_2$. 
And $T_1$ and $T_2$ have no common leaves. 
Therefore, \yuki{$T'S_i = T|_{\{a_i\} \cup (L[b]\cap L[T'])}$}, and $T'S_iR_j = T|_{\{a_i, b_j \}}$. 
So $T'S_iR_j$ is the single cherry $(a_i,b_j)$.
This cherry is survivably reducible with respect to both sequences $(a_i,b_j)$ and $(b_j,a_i)$, and the tree $T$. 
By construction,~$D_{i,j} = S_iR_j(b_j,a_i)$ or $D_{i,j} = S_iR_j(a_i,b_j)$.
It follows that~$T'$ is survivably reducible with respect to $D_{i,j}$ and $T$.

\yuki{Therefore, if~$s(b)\ge 2$, the set~$\{D_{i,j}: i\in[c'(a)],j\in[c'(b)]\}$ of sequences of~$T$ is a survival covering set of~$\mathcal{T}_T$.
It follows that~$c'(\rho)\le c'(a)c'(b)$.}

\paragraph{Case 2: $s(b) = 1$.} 
The set of sequences $\{S_1(b,a_1), S_2(b,a_2),\ldots,S_{c'(a)}(b,a_{c'(a)}),$ $ S_1(a_1,b)\}$ survivably covers $\mathcal{T}_T$, because the first $c'(a)$ sequences survivably reduces all elements in $\mathcal{T}_T$ except the subtree on the single leaf $b$, and the last sequence survivably reduces this subtree.
So we have $c'(\rho)\le c'(a) + 1$. 

\paragraph{Case 3: $s(a) = 1$.} The tree~$T$ is a cherry on~$\{a,b\}$. This can be survivably covered by $\{(a,b), (b,a)\}$, 
and it cannot be survivably covered by any set with just one sequence. 
So in this case,~$c'(\rho) = 2$.
\medskip

\noindent \rev{In all three cases the constructed sequence is a sequence of $T$ (the length of the sequences is exactly $|T| - 1$).}
Combining the three cases, we have
\[ c'(\rho) \le \begin{cases} 
      c'(a)c'(b) & \text{if } s(b) \ge 2 \\
      c'(a) + 1 & \text{if } s(b) = 1 \rev{\text{ and } s(a) > 1} \\
      2 & \text{if } s(a) = 1. \\
   \end{cases}
\]

\end{proof}

\begin{theorem}\label{thm:surviving_covering_number}
    Let $T$ be a binary tree \yuki{on $X$ \rev{rooted at $\rho$}, where $|X|\ge 2$}.
    Then $c'(\rho) = |\psi_\rho|$.
\end{theorem}

\begin{proof}
    We use induction on the number of leaves in $T$.
    Take a cherry $T$, and let~$a,b$ denote the children of~$\rho$. This can be survivably covered with two sequences $(a,b)$ and $(b,a)$. And $\psi_\rho = \{T[a], T[b]\}$. So $c'(\rho) = |\psi_\rho| = 2$, so the base case holds.

    Suppose the claim holds for every tree on at most $k$ leaves, where $k\ge2$. We show that the claim holds for trees on $k+1$ leaves.
    Take any tree $T$ on $k+1$ leaves, and let~$a,b$ denote the children of~$\rho$.
    There are two cases.

    \paragraph{Case 1: $s(b) \ge 2$.}  
    By the induction hypothesis, $c'(a) = |\psi_a|$ and $c'(b) = |\psi_b|$. By \Cref{lem:lower_bound}, $c'(\rho) \ge |\psi_\rho| \ge |\psi_a||\psi_b|$ and by \Cref{lem:upper_bound} $c'(\rho) \le c'(a)c'(b)$. So $c'(a)c'(b) \ge c'(\rho) \ge |\psi_\rho| \ge |\psi_a||\psi_b| = c'(a)c'(b)$, which means $c'(\rho) = |\psi_\rho|$.

    \paragraph{Case 2: $s(b) = 1$.}
    By the induction hypothesis, $c'(a) = |\psi_a|$.
    By \Cref{lem:lower_bound}, $c'(\rho) \ge |\psi_\rho| \ge |\psi_a| + 1$ and by \Cref{lem:upper_bound}, $c'(\rho) \le c'(a) + 1$.
    So $c'(a) + 1 \ge c'(\rho) \ge |\psi_\rho| \ge |\psi_a| + 1 = c'(a) + 1$.
    Hence, $c'(\rho) = |\psi_\rho|$.
\end{proof}

\begin{corollary}\label{cor:surviving_covering_number}
    Let $T$ be a binary tree on $X$ rooted at $\rho$ with children $a$ and $b$, such that $s(a) \ge s(b)$. Then 
\[ c'(\rho) = \begin{cases} 
      c'(a)c'(b) & \text{if } s(b) \ge 2 \\
      c'(a) + 1 & \text{if } s(b) = 1 \rev{\text{ and } s(a) > 1}\\
      2 & \text{if } s(a) = 1. \\
   \end{cases}
\]
\end{corollary}

Now we prove a recursive formula for the covering number. First, we prove that the covering number of a vertex is at least as much as the survival covering number of one of its children (\Cref{lem:trees_need_to_survive}). Then we prove the recursive formula (\Cref{thm:CoveringNumber}). 

\begin{lemma}\label{lem:trees_need_to_survive}
    Let $T$ be a binary tree on $X$, and let~$a,b$ be the children of~$\rho$. 
    Then $c(\rho) \ge c'(a)$.
\end{lemma}

\begin{proof}
    Take the set $\psi_a = \{T_1, T_2, \ldots, T_{c'(a)}\}$ \yuki{which are pairwise survivably irreducible. Such a set exists by \Cref{thm:surviving_covering_number}}. 
    Consider the set $\mathcal{Q} = \{T|_{L[T_1] \cup \{b_1\}},$\\$ T|_{L[T_2] \cup \{b_1\}}, \ldots, T|_{L[T_{c(a)'}] \cup \{b_1\}}\}$, where $b_1 \in L[b]$. By \Cref{lem:non-reducible-concat}, no pairs of trees in $\mathcal{Q}$ can be reduced by the same sequence. So we need at least $c'(a)$ sequences to cover all subtrees of $T$. In other words, we need $c(\rho) \ge c'(a)$.
\end{proof}


\thmCoveringNumber*

\begin{proof}[Proof of \Cref{thm:CoveringNumber}]
    Consider the tree $T$, \yuki{where the root~$\rho$ has two children~$a$ and~$b$}.
    We split into three cases. 
    The first case is when $s(a) = 1$. In this case, we have a cherry $(a,b)$. This can be covered with a single sequence, $(a,b)$.

    The second case is when $s(b) = 1$. Consider $T$. Let the survival covering set of $T[a]$ be $\mathcal{S} = \{S_1,S_2,\ldots, S_{c'(a)}\}$, with surviving leaf $a_1, a_2,\ldots, a_{c'(a)}$ respectively.
    Construct the set $\mathcal{Q} = \{S_1(b, a_1), S_2(b, a_2),\ldots S_{c'(a)}(b, a_{c'(a)})\}$. We claim that $\mathcal{Q}$ is an optimal covering set of $T$.
    Let $T'$ be a subtree of $T$. Take the tree $T''$ obtained by removing the leaf $b$ from $T'$. By definition, $T''$ is survivably covered by a sequence $S_i$ from $\mathcal{S}$. If $T' \cong T''$, then $T'$ is also covered by $S_i$ and thus by $S_i(b,a_i)$. If $T' \not\cong T''$, then $T'S_i$ is a cherry $(a_i, b)$ and this cherry can be covered by $(b,a_i)$, so $T'$ is covered by $S_i(b,a_i)$.
    This means $\mathcal{Q}$ is a covering set of $T$.
    Note that $c(\rho) \ge c'(a)$, because of \Cref{lem:trees_need_to_survive}.
    Because $|\mathcal{Q}| = |\mathcal{S}| = c'(a)$, $\mathcal{Q}$ is optimal.

    The third case is when $s(b) \ge 2$. By \yuki{applying} \Cref{lem:non-reducible-concat} \yuki{to $\psi_a\times \psi_b$}, we need at least $c'(a)c'(b)$ sequences. By \Cref{lem:upper_bound}, \yuki{all subtrees of~$T$} can be survivably covered by $c'(a)c'(b)$ sequences.
    And so it can also be covered by $c'(a)c'(b)$ sequences.

    Combining the above, we have \[ c(\rho) = \begin{cases} 
      c'(a)c'(b) & \text{if } s(b) \ge 2 \\
      c'(a) & \text{if } s(b) = 1 \rev{\text{ and } s(a) > 1}\\
      1 & \text{if } s(a) = 1. \\
   \end{cases}
\]
\end{proof}

\begin{figure}[h]
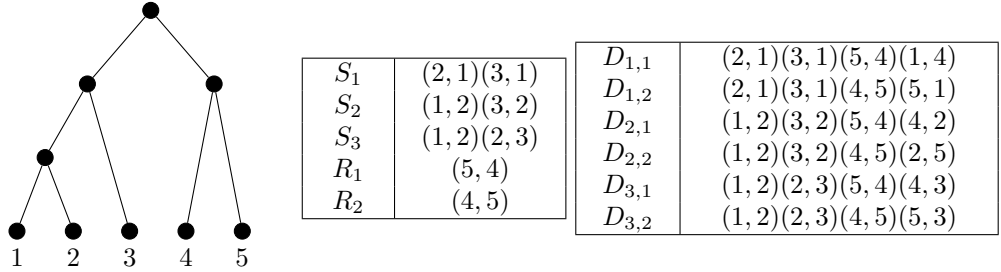

    \centering
    \begin{minipage}{0.3\textwidth}
        \centering
        \treeSequence{0.65}
    \end{minipage}
    \begin{minipage}{0.25\textwidth}
        \centering
        \coveringSequencesOriginal
    \end{minipage}
    \begin{minipage}{0.4\textwidth}
        \centering
        \coveringSequencesCombined
    \end{minipage}
    \caption{Example of constructing~$D_{i,j}$ in \Cref{lem:upper_bound}.
    Given is a tree on~$\{1,2,3,4,5\}$ with~$L[a] = \{1,2,3\}$ and~$L[b] = \{4,5\}$.
    The table in the middle shows~$S_i$ and~$R_j$ for~$i\in[3]$ and~$j\in[2]$; $S_i$ gives a sequence which reduces~$T[a]$ where~$i$ is the surviving leaf.
    In this case,~$P_3$ is the set of subtrees of~$T[a]$ induced by\rev{~$\{3\}$}, $\{2,3\}$, and $\{1,2,3\}$.
    For the table on the right, we use the function~$Ord_a(i) = i;\quad Ord_b(j) = j-3$.
    As an example, since~$Ord_a(2) = 2$ and~$Ord_b(4) = 1$, we have~$D_{2,1} = S_2R_1(4,2)$.
    }
    \label{fig:Sur_upper_bound}
\end{figure}


\section{Survival covering number for non-binary trees}\label{sec:surviving-covering-num-non-bin}

In this section, we try to generalize the results of \Cref{sec:covering-num-binary} for non-binary trees. We prove a recursive formula for the survival covering number in \Cref{thm:nonbin_surviving_covering_number}. To do this, we first prove lemmas that exploit the structure of the binary refinements of trees. However, finding a formula for the (non-surviving) covering number seems very challenging and it is addressed in \Cref{sec:NonBinCover}.

\begin{lemma}\label{lem:pairwise-irreducible-subtree-lower-bound-non-binary}
    Let $T$ be a non-binary tree on $X$ rooted at $\rho$.
    We partition the children of~$\rho$ into two sets.
    Let~$A', A''$ be the set of non-leaf and leaf children of~$\rho$, respectively.  
    Then
     \[ |\psi_\rho| \ge \begin{cases} 
      \prod _{v \in A'}{|\psi_v|} + |A''| & \text{if } A' \neq \emptyset \\
      |A''| & \text{if } A' = \emptyset. \\
   \end{cases}
\]
\end{lemma}

\begin{proof}
    We split into two cases, when $A' \neq \emptyset$ and $A' = \emptyset$.
    \paragraph{Case 1: $A' \neq \emptyset$.}
    We prove this case by constructing a pairwise \balint{survivably} irreducible subtree set of size $|A''| + \prod_{v\in A'}|\psi_v|$.
    Let $\psi_\rho' = \bigtimes_{x \in A'}\psi_x$.
    Let $P,Q \in \psi_\rho'$, such that $P \not\cong Q$. Then there exists a~$v\in A'$ where $P'$ and $Q'$ are subtrees of $P$ and $Q$ respectively, and $P', Q' \in \psi_v$.
    Therefore, $P$ and $Q$ are pairwise survivably \balint{irreducible} by \Cref{lem:non-reducible-concat} (noting that pairwise irreducibility implies pairwise survivably irreducibility). 
    So $\psi_\rho'$ is a pairwise \balint{survivably irreducible} subtree set with size $\prod_{x\in A'}|\psi_x| $.
    
    Let $\psi = \psi_\rho' \cup A''$. Take two subtrees $P, Q \in \psi$. There are two cases. The first case is when $P, Q \in \psi_\rho'$. In this case $P$ and $Q$ are pairwise \balint{survivably irreducible}, since $\psi_\rho'$ is a pairwise \balint{survivably irreducible subtree set}.
    The second case is $|\psi_\rho' \cap \{P,Q\}| \le 1$. 
    Then either $P \in A''$ or $Q \in A''$. Therefore, either $P$ or $Q$ is a single leaf. And note that it does not appear in any of the other trees in $\psi$. So, $P$ and $Q$ are pairwise \balint{survivably irreducible} by \Cref{lem:non-reducible-intersect}.
    Finally, if both~$P,Q\in A''$, then the claim follows immediately, since only one leaf can survive per sequence.
    This means that $\psi$ is a pairwise \balint{survivably irreducible subtree set}. For the size, we have $|\psi| = |A''| + \prod_{v\in A'}|\psi_v|$, and the claim follows.
    
    \paragraph{Case 2:  $A' =\emptyset$.} This case immediately follows from the fact that every element of $A''$ is a singleton and \Cref{lem:non-reducible-intersect}.
\end{proof}

\begin{figure}
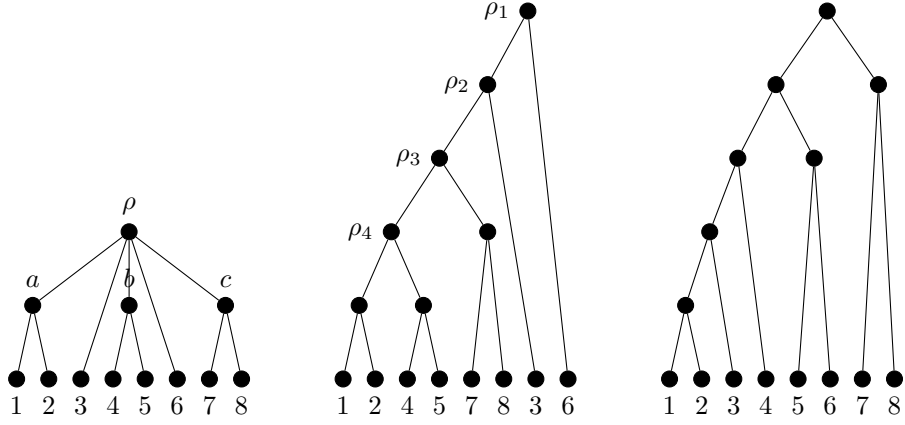

    \centering
    \begin{minipage}[b]{0.3\textwidth}
        \centering
        \binaryRefinementOriginal{0.25}
    \end{minipage}
    \begin{minipage}[b]{0.3\textwidth}
        \centering
        \biniaryRefinementOptimal{0.25}
    \end{minipage}
    \begin{minipage}[b]{0.3\textwidth}
        \centering
        \binaryRefinementSuboptimal{0.25}
    \end{minipage}
    \caption{Illustration for proof of \Cref{thm:nonbin_surviving_covering_number}. A non-binary phylogenetic tree~$T$ (on the left) with survival covering number\rev{~$c'(T)$}, and two of its binary refinements. One of the refinements (middle) has survival covering number exactly equal to \rev{~$c'(T)$}.
    This refinement is the caterpillarization (of $\rho$) of the left tree with respect to the ordering \rev{$(6,3,c,b,a)$} where~$a,b,c$ are the parents of \rev{$1,2; 4,5; 7,8$}, respectively.
    The other refinement (right), which is not a caterpillarization, has survival covering number strictly larger than\rev{~$c'(T)$}.}
    \label{fig:non-binary-and-its-refinements}
\end{figure}

Let $T$ be a non-binary phylogenetic tree. Let $r\in V(T)$ be a vertex with out-degree at least three, and let $a$ and $b$ be arbitrarily chosen children of $r$. 
We call the following operation a \emph{refinement} at $\{a,b\}$ on $T$: 
\begin{enumerate}
    \item remove edges $(r, a)$ and $(r, b)$ from $T$; and
    \item add a new vertex $r'$ and edges $(r, r'), (r', a)$ and $(r', b)$ to $T$.
\end{enumerate}
If a binary phylogenetic tree $T'$ is obtained by some refinements on a non-binary phylogenetic tree $T$, we call $T'$ a \emph{binary refinement} of $T$ (See \Cref{fig:non-binary-and-its-refinements}). 

\begin{lemma}\label{lemma:subset-of-binary-refinement-is-also-a-binary-refinement}
Let $T$ be a non-binary tree on $X$. Let $T'$ be a binary refinement of $T$. Let $X'\subseteq X$. Then $T'|_{X'}$ is a binary refinement of $T|_{X'}$.
\end{lemma}

\begin{proof}
    Since $T'$ is a binary refinement of $T$, there exists a finite sequence of trees $T = T_0, T_1, \dots, T_k = T'$, where each $T_i$ is obtained by applying a refinement operation to $T_{i-1}$. 
    Then, it is enough to prove that either $T_i |_{X'} \cong T_{i-1}|_{X'}$ holds or $T_i |_{X'}$ is obtained by applying a refinement operation to $ T_{i-1}|_{X'}$, for any $i\in [k]$.

    Suppose that $T_i$ is obtained by the refinement at $\{a, b\}$ on $T_{i-1}$. 
    Let~$r'$ be the parent of~$a,b$ in~$T_i$ and let~$r$ be the parent of~$a,b$ (and potentially other vertices) in~$T_{i-1}$.
    If $L[a]\cap X' = \emptyset$, then $b$ is the only child of $r'$ in $T_i|_{X'}$. Then $r'$ is suppressed by definition and $T_i|_{X'} \cong T_{i-1}|_{X'}$ holds. The analogous argument holds when $L[b]\cap X' = \emptyset$. 

    Then we consider $L[a]\cap X' \neq\emptyset$ and $L[b]\cap X' \neq \emptyset$. If there exists at least one child $c$ of $r$ other than $a$ or $b$ which satisfies $L[c]\cap X' \neq \emptyset$, then $r$ has out-degree at least three in $T_{i-1}|_{X'}$. 
    In $T_{i}|_{X'}$, $r$ has as children $r'$ and $c$ (and potentially more), and $r'$ has as children $a$ and $b$. Since the connectivity of the other vertices remains unchanged between $T_{i-1}|_{X'}$ and $T_{i}|_{X'}$, $T_{i}|_{X'}$ is obtained by applying the refinement at $\{a, b\}$ on $T_{i -1}|_{X'}$. If there does not exist such a child $c$ of $r$, then the children of $r$ except $r'$ is removed when restricting the leaf-set from $X$ to $X'$, and hence $r$ is suppressed. 
    Therefore, by the facts that the parent of $a$ and $b$ is $r$ on $T_{i-1}|_{X'}$ and $r'$ on $T_{i}|_{X'}$, we obtain $T_i|_{X'} \cong T_{i-1}|_{X'}$. This completes the proof.
\end{proof}

\begin{lemma}[Modified Lemma 3 of \cite{cherry_picking_and_network_containment}]\label{lem:sequence-of-binary-refinement-reduces-original-tree}
    Let $T$ be a non-binary tree on $X$. Let $T'$ be a binary refinement of $T$. If a sequence $S$ \rev{is a sequence of} $T'$, it also \rev{a sequence of} $T$.
\end{lemma}

    
    

\begin{lemma}\label{lem:cover-non-binary-with-binary-refinement}
    Let $T$ be a non-binary tree on $X$. Let $T'$ be a binary refinement of $T$. If $\mathcal{S}$ is a survival covering set of $T'$, then it is also a survival covering set of $T$.
\end{lemma}

\begin{proof}
    Let $X' \subseteq X$. There is an $S \in \mathcal{S}$ that survivably reduces $T'|_{X'}$ with respect to $T'$. 
    \rev{Note that~$S$ is a sequence of~$T'$ since survival covering sets contain elements which are sequences of the target tree.}
    And since $T'|_{X'}$ is a binary refinement of $T|_{X'}$ by \Cref{lemma:subset-of-binary-refinement-is-also-a-binary-refinement}, and by \Cref{lem:sequence-of-binary-refinement-reduces-original-tree}, $S$ also survivably reduces $T|_{X'}$ with respect to $T$. \rev{Note that the surviving leaf of $T$ is the same as the surviving leaf of $T'$ with respect to $S$ by the definition of the surviving leaf and by the fact that every element of $S$ affects both $T$ and $T'$.}
\end{proof}

The \emph{height} of a non-binary tree is the length of the longest path from the root to a leaf.

\paragraph{Caterpillarization.} Consider a non-leaf vertex $v$ of a tree $T$. Let $p$ be the parent of $v$ (if it exists) and $a_1, a_2, \ldots, a_k$ the children of $v$ such that $s(a_1) \le s(a_2) \le \ldots\le s(a_k)$. We call the following operation the \emph{caterpillarization} of vertex $v$ with respect to the ordering~$(a_1,a_2,\ldots, a_k)$.
See the left and middle trees in~\Cref{fig:non-binary-and-its-refinements} for an example of caterpillarization. 
\begin{itemize}
    \item Delete $v$ from the tree, and create $k-1$ new vertices $v_1, \ldots, v_{k-1}$.
    \item Add an edge from $p$ to $v_1$.
    \item Add an edge from $v_i$ to $v_{i+1}$ for all $i\in [k-2]$.
    \item Add an edge from $v_i$ to $a_i$ for all $i\in [k-1]$, and add an edge from $v_{k-1}$ to $a_k$.
\end{itemize}
The \emph{caterpillarization} of a tree $T$ is a tree obtained by applying the caterpillarization operation to every non-leaf vertex of $T$.
Furthermore, we call $v_1, \ldots, v_{k-1}$ the \emph{refined path} of $v$, and $v_1$ the \emph{head} of $v$.
We do not apply the caterpillarization processes to leaves, but for consistencies sake, we say that the \emph{head} of a leaf~$a$ is~$a$ itself.

\begin{observation}\label{obs:caterpillarization}
    A caterpillarization $R$ of a tree $T$ is also a binary refinement of $T$.
\end{observation}
\begin{proof}
    The caterpillarization of a vertex can be written as a series of taking refinements.
    These refinement steps convert non-binary vertices into a chain of binary vertices.
\end{proof} 

\begin{theorem}\label{thm:nonbin_surviving_covering_number}
    Let $T$ be a non-binary tree on $X$, where $|X|\ge 2$, and let~$\rho$ be the root of~$T$. Let $R$ be a caterpillarization of $T$.
    Then $c'(T)= c'(R) = |\psi_\rho|$.
    Let the set of children of $\rho$ be $A$ \yuki{where~$|A|\ge 2$}. Let $A' = \{v | v \in A, s(v)\ge2\}$, and let $A'' = A \setminus A'$.
    Furthermore, we have
\[ c'(T) = \begin{cases} 
      \prod _{v \in A'}{c'_T(v)} + |A''| & \text{if } A' \neq \emptyset \\
      |A''| & \text{if } A' = \emptyset. \\
   \end{cases}
\] 
\end{theorem}
\begin{proof}
    We use induction on the height~$h(T)$ of~$T$.
    For the base case, if~$h(T)=1$, then all leaves are attached to the root, and~$c'(T) \ge |X|$ since exactly one leaf can be the surviving leaf of a sequence.
    On the other hand, consider the caterpillarization~$R$ of~$T$.
    By \Cref{cor:surviving_covering_number},~$c'(R) = |X|$; by \Cref{lem:cover-non-binary-with-binary-refinement}, we have~$c'(T)\le c'(R) = |X|$.
    It follows that~$c'(T) = |X|$.
    To see that~$|\psi_\rho| = |X|$, note that one can take the set of singletons in~$X$, which is pairwise survivably irreducible.
    Thus,~$c'(T) = |\psi_\rho|$.

    Now suppose we have proved the claim for trees with height at most~$h$ where~$h\ge 1$. 
    We shall show that the claim holds for trees with height~$h+1$.
    Let $T$ be a tree of height $h+1$ with root $\rho$ and let $R$ be an arbitrary caterpillarization of $T$.
    Let $A = \{a_1, \ldots, a_k\}$ be the set of children of $\rho$.
    \rev{Let $\rho_1, \ldots, \rho_{k-1}$ be the refined path of $\rho$ in $R$.}
    Let $B = \{b_1, \ldots, b_k\}$ be the set of vertices where $b_i$ is the head of $a_i$ in $R$. 
    The elements of $A$ are vertices of $T$. The elements of $B$ and the vertices $\rho_1, \ldots, \rho_{k-1}$ are vertices of $R$ and may not necessarily be vertices of~$T$.
    As defined in the claim, let $A'$ and $A''$ be the set of non-leaf and leaf elements of $A$, respectively.
    Define $B'$ and $B''$ for $B$ analogously.
    
    The tree $R$ is a binary refinement of $T$ by \Cref{obs:caterpillarization}.
    By applying \Cref{cor:surviving_covering_number} to $R$, in particular by applying it to $\rho_{k-1}, \ldots, \rho_1$ repeatedly, we get 
     \[ c'(R) = \begin{cases} 
          \prod _{v \in B'}{c_R'(v)} + |B''| & \text{if } B' \neq \emptyset \\
          |\rev{B}''| & \text{if } B' = \emptyset, \\
       \end{cases}
    \]
    \rev{since in the caterpillarization process, we place all singletons at the top of the refined path.}

    Since caterpillarization is only applied to non-leaf vertices, we have~$|A''| = |B''|$.
    Note that every element $b$ of $B'$ is the head of exactly one element $a$ of $A'$. 
    Furthermore, for any~$a\in A'$, we have $h(T[a]) < h(T)$ because $a$ is a child of the root of $T$. 
    Since $b$ is the head of $a$, and $R$ is a caterpillarization of $T$, the subtree $R[b]$ is a caterpillarization of $T[a]$. By the induction hypothesis, $c'_R(b) = c'_T(a) = |\psi_a|$. It follows that
        
    \[ c'(R) = \begin{cases} 
          \prod _{v \in A'}{|\psi_v|} + |A''| & \text{if } A' \neq \emptyset \\
          |A''| & \text{if } A' = \emptyset. \\
       \end{cases}
    \]

    By \Cref{lem:pairwise-irreducible-subtree-lower-bound-non-binary}, we have 
    \[ |\psi_\rho| \ge \begin{cases} 
          \prod _{v \in A'}{|\psi_v|} + |A''| & \text{if } A' \neq \emptyset \\
          |A''| & \text{if } A' = \emptyset. \\
       \end{cases}
    \]

    Therefore, $c'(R) \le |\psi_\rho|$.
    By \Cref{lem:cover-non-binary-with-binary-refinement}, the covering set of a binary refinement of a tree is also a covering set of the original tree.
    By \Cref{obs:caterpillarization}, $R$ is a binary refinement of $T$, and therefore $c'(T) \le c'(R)$.
    
    It follows that $c'(T) \le |\psi_\rho|$.
    But we also have $c'(T) \ge |\psi_\rho|$, because the elements of $\psi_\rho$ are pairwise survivably irreducible.
    Combining the above, we have  $|\psi_\rho| \le c'(T) \le c'(R) \le |\psi_\rho|$. Therefore, $c'(T) = c'(R) = |\psi_\rho|$. Additionally, $R$ is a caterpillarization of $T$, so the claim follows.
\end{proof}

\section{Covering Number for non-binary trees}\label{sec:NonBinCover}

\yuki{In this section, we shall use both phylogenetic trees and unrooted trees, and use their full names to distinguish between the two objects.}
In this section, we \yuki{consider the covering number of non-binary phylogenetic trees.}
\medskip
\begin{center}
\noindent\fbox{\parbox{0.95\linewidth}{
{\sc CoveringNumber$(T,k)$}\\
{\bf Input:} Non-binary phylogenetic tree $T$ on $n$ leaves and an integer $k$.\\
{\bf Decide:} Does there exist a set of cherry-picking sequences of size at most $k$ that reduces all subtrees of $T$?}}
\end{center}
\medskip

\yuki{We write~$c(T)$ to denote the smallest~$k$ for which {\sc CoveringNumber$(T,k)$} is a YES-instance.}
We focus on \emph{star phylogenetic trees}, where every leaf is a child of the root.
Let us denote star phylogenetic trees on~$n$ leaves with~$T^*(n)$.
We show that finding $c(T^*(n))$ is equivalent to solving the following problem.

Let~$n,k\in \mathbb{N}$.
\yuki{By an \emph{unrooted tree}, we mean an undirected connected acyclic graph.}
Suppose we have a set $S=\{T_1,\ldots, T_k\}$ of unrooted trees, 
where each unrooted tree~$T_i$ has~$n$ vertices, which are labelled from $1$ to $n$.
Let $T = (V,E)$ be an unrooted tree. We write~$|T| = |V|$.
Let~$A \subseteq V$. 
We write $T[A]$ to denote the \yuki{\emph{induced subgraph} of~$T$, which is the graph whose vertex set is~$A$ and whose edge set consists of all elements in~$E$ with both endpoints in~$A$.} 
We say that $S$ is \emph{complete for~$n$} if for any subset~$A$ of~$[n]$, there exists an unrooted tree~$T_i\in S$ such that~$T_i[A]$ is connected.
\medskip
\begin{center}
\noindent\fbox{\parbox{0.95\linewidth}{
{\sc SubsetConnectivity}\\
{\bf Input:} Natural numbers $n$ and $k$.\\
{\bf Decide:} Does there exist a set of unrooted trees $\{T_1,\ldots, T_k\}$ that is complete for $n$?}}
\end{center}
\medskip

\yuki{We show that \textsc{CoveringNumber} for star phylogenetic trees and \textsc{SubsetConnectivity} are equivalent.}

\begin{theorem}\label{thm:Equiv}
Let~$n,k\in\mathbb{N}$. Then~$(T^*(n),k)$ is a YES-instance of~\textsc{CoveringNumber} if and only if~$(n,k)$ is a YES-instance of~\textsc{SubsetConnectivity}.
\end{theorem}

\begin{proof}
    Suppose first that $(T^*(n),k)$ is a YES-instance of \textsc{CoveringNumber}. 
    Then we have at most~$k$ cherry-picking sequences~$S_1,\ldots, S_k$ which cover~$T^*(n)$.
    We construct $k$ unrooted trees on~$n$ vertices. 
    For each $S_i$, take the following steps.
    \begin{itemize}
        \item Initialize an empty graph $T_i$ on $n$ vertices. 
        \item For every element $(x_j,y_j)$ in $S_i$, add an edge between $x_j$ and $y_j$ in $T_i$. 
    \end{itemize}
    
    By definition of cherry picking sequences, $S_i$ has $n-1$ distinct elements, and each element adds an edge in $T_i$. 
    We prove that~$T_i$ is a tree by showing that it is acyclic.
    
    Suppose for a contradiction that~$T_i$ has a cycle $a_1, a_2, \ldots, a_m$.
    Because of the construction of $T_i$, the elements $(a_1,a_2), \ldots, (a_{m-1}, a_m), (a_m, a_1)$ must appear in $S_i$ in some order (with possibly swapping the coordinates in some pairs).
    Moreover, by definition, an element which appears as a first coordinate in $S_i$ cannot appear later in $S_i$, neither as a first coordinate, nor as a second coordinate.
    So the sequence $S_i$ must contain the pairs $(a_1, a_2), \ldots, (a_{m-1}, a_m), (a_m, a_1)$ or the pairs $(a_2, a_1), (a_3, a_2), \ldots, (a_m, a_{m-1}), (a_1, a_m)$. In either cases, consider the pair among these pairs that appears the latest in $S_i$. Its second element must have appeared as a first coordinate in $S_i$ in an earlier pair due to the previous observation, which contradicts the definition of $S_i$. So $T_i$ cannot contain a cycle.

    It follows that $T_i$ is an unrooted tree, since it has $n$ vertices, $n-1$ edges, and it is acyclic.
    \medskip
    
    It remains to show that $\{T_1, \ldots, T_k\}$ is complete for~$n$.
    Consider any subset~$A\subseteq [n]$, where~$|A|\ge 2$.
    Since $\{S_1,\ldots, S_k\}$ covers the set of all phylogenetic subtrees of $T^*(n)$, there must be a sequence $S_i$ which reduces $T^*(n)|_A$. 
    By \Cref{lem:Subsequence}, $S_i$ has a subsequence $S'_i$ of size $|A|-1$, which contains only the elements of $A$.
    The unrooted tree $T_i[A]$ is acyclic, since $T_i$ is a tree.
    Furthermore, $T_i[A]$ contains~$|A|$ vertices, and has~$|A|-1$ edges.
    Then it is an unrooted tree and it must be connected.
    So the set of unrooted trees~$\{T_1,\ldots,T_k\}$ is complete for~$n$, and thus $(n,k)$ is a YES-instance of \textsc{SubsetConnectivity}.
    \medskip

    Now suppose that $(n,k)$ is a YES-instance of \textsc{SubsetConnectivity}. 
    In particular, let $\{T_1, \ldots, T_k\}$ be a set of unrooted trees that is complete for $n$.
    We show that $(T^*(n), k)$ is a YES-instance of \textsc{CoveringNumber}.

    We construct a cherry-picking sequence $S_i$ from $T_i$ for all $i \in [k]$ as follows.
    \begin{enumerate}
        \item Initialize $S_i$ to be an empty sequence.
        \item While $T_i$ contains at least $2$ vertices:
        \begin{enumerate}
            \item Find two vertices $x$ and $y$, where $x$ is a degree-1 vertex and $y$ its neighbour.
            \item Append $(x,y)$ to $S_i$.
            \item Remove $x$ from $T_i$.
        \end{enumerate}
    \end{enumerate}


    Now we show that $S_i$ is a cherry-picking sequence for $T^*(n)$.
    Note that $|S_i| = n - 1$, because $|T_i| = n$ and we add an element to $S_i$ for each vertex of $T_i$ except one. 
    It follows from the process, that $S_i$ has all elements of $[n]$ as a first coordinate except one, and if an $x \in [n]$ appears as a first coordinate, it will not appear later in $S_i$.
    And since every leaf of $T^*(n)$ is connected to the root, every pair $(x,y)$, where $x \neq y$, is a cherry as long as both $x$ and $y$ are not removed from $T^*(n)$.
    Therefore, every element of $S_i$ removes a leaf from $T^*(n)$, and since $|S_i| = n-1$,  $S_i$ reduces $T^*(n)$.

    Let $A \subseteq [n]$, where $|A|\ge2$. Since the set $\{T_1, \ldots, T_k\}$ is complete for $n$, there is an $i \in [k]$ such that $T_i[A]$ is connected. We claim that there is a subsequence $S_i'$ in $S_i$, where $|S_i'| = |A|-1$, and $S_i'$ contains only the elements of $A$ as first and second coordinates.

    To see this, first observe that there is a one-to-one correspondence between the edges of $T_i$ and the elements of $S_i$. In particular, for any $x,y\in [n]$, we have that $xy$ is an edge in $T_i$ if and only if $(x,y)$ or $(y,x)$ is an element of $S_i$.
    The edges in~$T_i[A]$ are also edges in~$T_i$. 
    All other edges in~$T_i$ have at least one endpoint from~$[n]\setminus A$.
    Since~$T_i[A]$ is an unrooted tree,~$T_i$ contains exactly~$|A|-1$ edges where both endpoints are in~$A$.
    It follows that~$S_i$ has a subsequence~$S'_i$ where~$S'_i = |A|-1$, and $S_i'$ contains only the elements of $A$ as first and second coordinates.
    

    By the above claim and \Cref{lem:Subsequence}, if $T_i[A]$ is connected, then $S_i$ reduces $T^*(n)|_A$. This means that since $\{T_1, \ldots, T_k\}$ is complete for $n$, for every phylogenetic subtree $T'$ of $T^*(n)$, there is a sequence $S\in\{S_1, \ldots, S_k\}$ that reduces $T'$.
    Therefore, if $(n,k)$ is a YES-instance of \textsc{SubsetConnectivity}, then $(T^*(n), k)$ is a YES-instance of \textsc{CoveringNumber}.

    Since we showed both directions of the claim, $(T^*(n),k)$ is a YES-instance of \textsc{CoveringNumber}, if and only if $(n,k)$ is a YES-instance of \textsc{SubsetConnectivity}.
\end{proof}

In the following theorems we prove \rev{relationships between consecutive elements, i.e.,~$c(n)$ and~$c(n+1)$ for~$n\ge 2$} for the \textsc{SubsetConnectivity} problem.
Let $star_{n,k}$ denote the unrooted tree on $n$ vertices, where the vertices are labeled from $1$ to $n$, and the vertex labeled $k$ is adjacent to every other vertex.
In the rest of the section, we write~$c(n)$ to be the smallest integer where~$\{T_1,\ldots, T_{c(n)}\}$ is complete for~$n$.
See \Cref{fig:optimalDoubleStar} for an illustration of three cases~$c(4) = 2, c(5) = 3,$ and~$c(6) = 4$.

\begin{figure}[h]
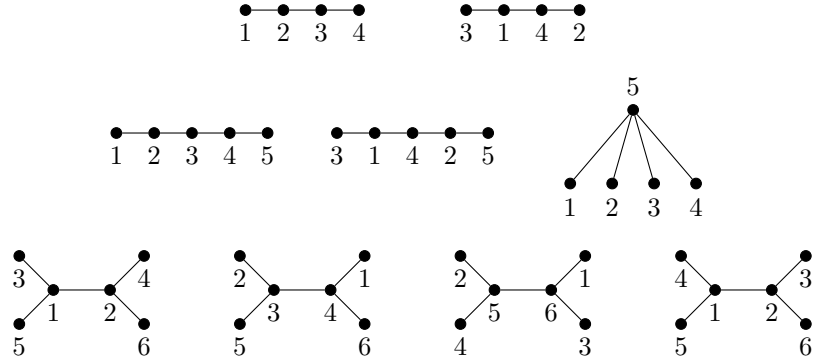

    \centering
    \begin{minipage}{0.2\textwidth}
        \centering
        \pathFour{1}{2}{3}{4}{0.5}
    \end{minipage}
    \begin{minipage}{0.2\textwidth}
        \centering
        \pathFour{3}{1}{4}{2}{0.5}
    \end{minipage} \\
    \vspace{0.3cm}
    \begin{minipage}{0.2\textwidth}
        \centering
        \subconFive{1}{2}{3}{4}{5}{0.5}
    \end{minipage}
    \begin{minipage}{0.2\textwidth}
        \centering
        \subconFive{3}{1}{4}{2}{5}{0.5}
    \end{minipage}
    \begin{minipage}{0.2\textwidth}
        \centering
        \starTreeFive{0.5}
    \end{minipage} \\
    \vspace{0.3cm}
    \begin{minipage}{0.2\textwidth}
        \centering
        \doubleStarSix{1}{3}{5}{2}{4}{6}{0.5}
    \end{minipage}
    \begin{minipage}{0.2\textwidth}
        \centering
        \doubleStarSix{3}{2}{5}{4}{1}{6}{0.5}
    \end{minipage}
    \begin{minipage}{0.2\textwidth}
        \centering
        \doubleStarSix{5}{2}{4}{6}{1}{3}{0.5}
    \end{minipage}
    \begin{minipage}{0.2\textwidth}
        \centering
        \doubleStarSix{1}{4}{5}{2}{3}{6}{0.5}
    \end{minipage}
    \caption{(Top figures) A set of unrooted trees that gives the optimal answer $c(4) = 2$ for the \textsc{SubsetConnectivity} problem. 
    (Middle figures) A set of unrooted trees that gives the \rev{optimal} answer $c(5) = 3$.
    (Bottom figures) A set of unrooted trees that gives the optimal answer $c(6) = 4$.}
    \label{fig:optimalDoubleStar}
\end{figure}

\begin{theorem}\label{thm:nostar}
    Let~$n\ge 2$. 
    If~$c(n) = c(n+1)$, then every minimum complete set of trees for~$n+1$ does not contain $star_{n+1,m}$ for all $m \in [n+1]$.
\end{theorem}

\begin{proof}

    Suppose for a contradiction that there is a complete set $S$ for $n+1$ that contains a $star(n+1,m)$. Without loss of generality, let $m = n+1$.
    Note that 
    $star(n+1,n+1)$ with the vertex~$n+1$ removed (and the incident edges removed)
    consists of only isolated vertices.
    Therefore, for all sets $A$ such that $n+1 \notin A$ and $|A| \ge 2$, we have that $star(n+1,n+1)[A]$ is disconnected. So, every such set must be connected in some $T \in S \setminus \{star(n+1,n+1)\}$. And since the empty set and the single element sets are connected in every tree, $S \setminus \{star(n+1,n+1)\}$ is complete for $n$.
    Let $S'$ contain all elements of $S \setminus \{star(n+1,n+1)\}$, 
    where in each tree, the vertex labeled $n+1$ is removed from the tree and the resulting components are connected arbitrarily to form a new tree.
    Observe that since $S \setminus \{star(n+1,n+1)\}$ is complete for $[n]$, $S'$ is also complete for $[n]$.
    Since $|S'| = |S \setminus \{star(n+1,n+1)\}| = |S| - 1$, we have $c(n) \le c(n+1) - 1$, which contradicts the assumption that $c(n) = c(n+1)$.
    The claim holds.
\end{proof}

\begin{theorem}
\label{thm:cn-plus-one-or-cn}
    $c(n+1) = c(n)+1$ or $c(n)$.
\end{theorem}
\begin{proof}
 
    First, we prove that $c(n+1) \le c(n) + 1$. Let $\{T_1, \ldots, T_{c(n)}\}$ be a set of unrooted trees which is complete for $n$. Let $G = \{T'_1, \ldots, T'_{c(n)}\}$, where $T'_i$ is obtained by adding the edge $(1,n+1)$ to $T_i$. Note that $G$ is complete for $n$. 
    Note that for every subset \rev{$S\subseteq [n+1]$ such that $n+1 \in S$}, the unrooted tree $star_{n+1, n+1}[S]$ is connected.
    Therefore, $G \cup \{star_{n+1,n+1}\}$ is complete for $n+1$, and $|G \cup \{star_{n+1,n+1}\}| = c(n) + 1$. So $c(n+1) \le c(n) + 1$.

    Now we prove that $c(n) \le c(n+1)$. Let $\{T_1, \ldots, T_{c(n+1)}\}$ be a set of unrooted trees that is complete for $n+1$. Let $G = \{T'_1, \ldots, T'_{c(n+1)}\}$, where $T'_i$ is obtained by removing the vertex labeled $n+1$ from $T_i$ and arbitrarily connecting the resulting forest to form a new tree. $G$ is complete for $n$. To see this, for every subset $S \subseteq [n]$, there is a $T_i$ such that $T_i[S]$ is connected. One can verify that $T'_i[S]$ still remains connected.
    
    

    So $c(n) \le c(n+1) \le c(n)+1$, which means $c(n+1) = c(n)+1$ or $c(n+1) = c(n)$.
\end{proof}

\section{Reducing all subtrees by a single long sequence}\label{sec:Hybrid>cover}

In previous sections, we have restricted ourselves to cherry-picking sequences on the reference tree, i.e., sequences of length~$|T|-1$.
In general, one can have longer cherry-picking sequences; these no longer correspond to trees but to orchard networks (see \Cref{sec:TCOrch} or \cite{cherry_picking_and_network_containment}).
So we ask the following question: given a set of all subtrees of a tree~$T$, can we reduce them by a single sequence? If so, what is the minimum length of such a sequence?
The goal of this section is to show that there is always such a sequence, and that the minimum length is~$\binom{n}{2}$, where~$n$ is the number of leaves in the reference tree (\Cref{thm:OHSBound}).
In this section, all trees are phylogenetic trees.

\paragraph{Cherry order graph.}
Consider a tree $T$ on~$n$ leaves.
Let~$\cD(T)$ denote the set of all non-trivial subtrees of~$T$.
We call every pair~$\{x,y\}\subseteq L[T]$ a \emph{pseudo-cherry} of~$T$.
A \emph{cherry order graph} of $T$ is denoted by $G_T$ and constructed as follows. 
The vertices of $G_T$ are the pseudo-cherries of $T$. In particular, the vertex corresponding to the pseudo-cherry $\{x,y\}$ is denoted by $G_T\{x,y\}$. Note that $G_T$ has exactly $\binom{n}{2}$ vertices.
There is a directed edge from $G_T\{x,y\}$ to $G_T\{w,z\}$ if and only if $\LCA_T(w,z)$ is strictly above $\LCA_T(x,y)$ in $T$.

\rev{
\paragraph{Topological ordering.} Let $T$ be a tree on $n$ leaves. We say that a (not necessarily cherry-picking) sequence of ordered pairs contains a pseudo-cherry $\{x,y\}$ if it contains $(x,y)$ or $(y,x)$.
Let $S$ be a (not necessarily cherry-picking) sequence of ordered pairs.
We call $S$ a \emph{topological ordering of $G_T$}
if $S$ contains every pseudo-cherry of $T$ exactly once, and does not contain any other elements; and for all pairs of elements $(x,y)$ and $(w,z)$ of $S$, if $G_T$ has an edge from $G_T\{x,y\}$ to $G_T\{w,z\}$ then $(x,y)$ appears in $S$ before $(w,z)$.
We call $S'$ a \emph{partial topological ordering} of $G_T$ if it can be obtained by deleting zero or more elements of $S$.
}
\begin{observation}\label{obs:cogAcyclic}
    The cherry order graph $G_T$ of $T$ is acyclic.
    This also means that there is a topological ordering of $G_T$.
\end{observation}
\begin{proof}
    Consider two pseudo-cherries $\{x,y\} \ne \{w,z\}$. $\LCA_T(x,y)$ cannot be above and below $\LCA_T(w,z)$ in $T$ at the same time.
\end{proof}


\paragraph{\rev{Cherry order sequences. }}

    Consider a tree $T$ on $n$ leaves rooted at $\rho$ and a topological ordering $S_O$ of $G_T$. $S_O$ exists by \Cref{obs:cogAcyclic}. Note that $S_O$ has length $\binom{n}{2}$.

\rev{
    Let $a$ be an arbitrarily child of $\rho$. Let $A = \{a_1, \ldots, a_{L[a]}\} = L[a]$ be the set of leaves under $a$.
    Let $B = \{b_1, \ldots, b_{n - |L[a]|}\} = L[\rho] \setminus L[a]$ be the set of leaves not under $a$.
    We now construct a new sequence $S$ from $S_O$ by the following process.
    \begin{enumerate}
        \item[(S1)] Initialize $S := S_O$.
        \item[(S2)] Delete all elements $(x,y)$ from $S$ such that exactly one of $x$ and $y$ is in $A$ and the other is in $B$.
        \item[(S3)] Append $(a_2, b_1), \ldots,  (a_2,b_{|B|}),  \ldots, (a_{|A|}, b_1), \ldots, (a_{|A|}, b_{|B|})$ to the end of $S$.
        \item[(S4)] Append $(b_1, a_1), \ldots (b_{|B|}, a_1)$ to the end of $S$.
    \end{enumerate}
    We call the sequence $S$ produced by the above process a \emph{cherry order sequence} of $T$.
}

\rev{
\begin{observation}\label{obs:len-cherry-order-sequence}
    Let $T$ be a tree on $n$ leaves, and $S$ be a cherry order sequence of $T$. Then $S$ contains all pseudo-cherries of $T$ exactly once and has length $\binom{n}{2}$.
\end{observation}
\begin{proof}
    The elements deleted by (S2) from $S$ are exactly the ones added in (S3) and (S4) up to the order of their coordinates.
\end{proof}
}

\rev{
\begin{lemma}\label{lem:cherry-order-cherry-picking-sequence}
    Let $T$ be a tree and let $S$ be a cherry order sequence of $T$. Then $S$ is a cherry-picking sequence.
\end{lemma}
\begin{proof}
    We use the notation $A$, $B$, $\{a_1, \ldots, a_{|A|}\}$, and $\{b_1, \ldots, b_{|B|}\}$ from the construction of $S$.
    Consider the construction of $S$.
    In (S3), all elements of $\{a_2, \ldots, a_{|A|}\}$ are appended as a first coordinate of a pair at the end of $S$. They never appear as a second coordinate of a pair again.
    In (S4), all elements of $\{b_1, \ldots, b_{|B|}\}$ are appended as a first coordinate of a pair at the end of $S$.
    So, the only leaf that appears in $S$ as a second coordinate without appearing as a first coordinate of a pair later in $S$ is $a_1$, but $a_1$ is the second coordinate of the last pair.
    It follows by the definition of cherry-picking sequences that $S$ is a cherry-picking sequence.
\end{proof}
}

\rev{
\begin{lemma}\label{lem:cherry-order-topo-order}
    Let $T$ be a tree, and let $S$ be a cherry order sequence of $T$. Then $S$ is a topological ordering of $G_T$.
\end{lemma}
\begin{proof}
    We use the notation from the definition of cherry order sequences.
    We prove the claim by showing that after every step of the construction of $S$, the sequence is a (partial) topological ordering of $G_T$.
    After (S1), the sequence $S$ is a topological ordering of $G_T$, since $S_O$ is defined to be a topological ordering of $G_T$.
    In (S2), elements are deleted from $S$, no elements are added and $S$ becomes a partial topological ordering of $G_T$.
    Note that $G_T\{x,y\}$ does not have any outgoing edges for all $x \in A$ and $y \in B$, because $\LCA_T(x,y) = \rho$ is the root of $T$.
    Therefore, $S$ remains a partial topological ordering of $G_T$ after appending
    the element $(x,y)$ or $(y,x)$ in (S3) and (S4).
    After (S4), all pseudo-cherries of $T$ appear in $S$ exactly once by \Cref{obs:len-cherry-order-sequence}, so $S$ is a topological ordering of $G_T$.
\end{proof}
}

\begin{lemma}\label{lem:SReducesT}
    Let $T$ be a tree and $S$ be a cherry order sequence of $T$. Let $T' \in \cD(T)$.
    Then $S$ reduces $T'$.
\end{lemma}

\begin{proof}
    We prove the claim by induction on the number of pseudo-cherries of $T'$.
    For the base case, consider $T'$ to be a single pseudo-cherry $\{a,b\}$. By construction of $S$, either $(a,b)$ or $(b,a)$ appears in $S$, so $S$ reduces $T'$. Therefore, the base case holds.
    
    Now suppose that the claim holds for all subtrees on at most $k$ pseudo-cherries. We show that it holds for trees on $k+1$ pseudo-cherries.
    Let~$T'$ be a tree on~$k+1$ pseudo-cherries.
    Let $(a,b)$ be the first element of $S$ such that $a,b \in L[T']$. Note that $T'$ is not affected by any earlier element of $S$. We claim that $(a,b)$ affects $T'$. Suppose this is not the case. 
    Then there exists $c \in L[T']$ such that $LCA_{T'}(a,c)$ or $LCA_{T'}(b,c)$ is strictly below $LCA_{T'}(a,b)$. 
    Without loss of generality, suppose that $LCA_{T'}(a,c)$ is strictly below $LCA_{T'}(a,b)$. This means that there is an edge $(G_{T'}\{a,c\}, G_{T'}\{a,b\})$ in $G_{T'}$. And since $T'$ is a subtree of $T$, the edge $(G_T\{a,c\}, G_T\{a,b\})$ must also be in $G_T$. \rev{But this means that the pair $(a,c)$ or $(c,a)$ appears earlier in $S$ than $(a,b)$ because $S$ is a topological ordering of $G_T$ by \Cref{lem:cherry-order-topo-order}. This contradicts the choice of~$(a,b)$.}
    So $T'$ is affected by $(a,b)$, and therefore, $T'(a,b)$ is a tree on at most $k$ pseudo-cherries.
    
    Consider the two sequences $S'$ and $S''$, where $S'$ is obtained from $S$ by deleting the first elements of $S$ up to and including $(a,b)$ and $S''$ is obtained by deleting the last $|S'|$ elements of $S$. Note that $S = S''S'$. Note that by the choice of $(a,b)$, $T'S'' = T'(a,b)$. And also note that $S''$ does not affect $T'(a,b)$, 
    by choice of $(a,b)$.
    By the induction hypothesis, $S$ reduces $T'(a,b)$. And because $S''$ does not affect $T'(a,b)$, $T'(a,b)$ must be reduced by $S'$. And since $T'S'' = T'(a,b)$ and $S'$ reduces $T'(a,b)$, it follows that $T'$ must be reduced by $S = S''S'$.
\end{proof}

\begin{theorem}\label{thm:OHSBound}
    Let~$T$ be a tree on~$n$ leaves.
    Then a cherry-picking sequence that reduces all trees in~$\mathcal{D}(T)$ is of length at least~$\binom{n}{2}$, and this bound is tight.
\end{theorem}

\begin{proof}
    Let $S$ be a cherry-picking sequence of minimum length that reduces all trees in $\cD(T)$. 
    
    We first show that ~$|S| \ge \binom{n}{2}$.
    $S$ must reduce all pseudo-cherries. This means that every pseudo-cherry must appear in $S$ as an ordered pair, and there are $\binom{n}{2}$ pseudo-cherries in total. Therefore, ~$|S| \ge \binom{n}{2}$.

    \rev{
    Now consider a cherry order sequence $S'$ of $T$.
    By \Cref{lem:SReducesT}, $S'$ reduces $T$, and by \Cref{lem:cherry-order-cherry-picking-sequence} and \Cref{obs:len-cherry-order-sequence}, $S'$ is a cherry-picking sequence of length $\binom{n}{2}$.
    }
\end{proof}

    

\section{Discussion}\label{sec:Discussion}
We summarize our main contributions and outline several questions that remain open.

\paragraph{Summary of results.}
In this paper, we sought to answer the following question. `Given a phylogenetic tree~$T$, what is its covering number -- the minimum number of cherry-picking sequences of~$T$ needed to reduce all subtrees of~$T$?'.
A key notion used for this computation was the \emph{survival covering number} $c'(T)$ of a tree~$T$.
This is a variant of the covering number $c(T)$ that imposes, in addition to the covering condition, the requirement that every subtree of $T$ be a surviving subtree with respect to some sequence in the covering set.
For binary trees, we gave a recursive formula for $c'(T)$ (\Cref{thm:surviving_covering_number} and \Cref{cor:surviving_covering_number}) and used it to derive a recursive formula for the covering number $c(T)$ (\Cref{thm:CoveringNumber}).
\yuki{In the proof, we also gave an} explicit construction for building an optimal covering set of sequences from optimal covers of the subtrees rooted at the children.

For non-binary trees~$T$, the survival covering number $c'(T)$ remains \taka{tractable}.
\Cref{thm:nonbin_surviving_covering_number} gives a recursive formula in terms of the leaf and non-leaf children of the root.
In particular, one can compute~$c'(T)$ 
using an appropriately chosen binary refinement of $T$ (\Cref{lem:sequence-of-binary-refinement-reduces-original-tree,lem:cover-non-binary-with-binary-refinement}).
This allows the binary result to be generalized to the non-binary setting.

\taka{By contrast, the (non-survival) covering number $c(T)$ of a non-binary tree~$T$ appears to be more difficult even for star trees.}
In \Cref{sec:NonBinCover}, we showed that determining $c(T^{\ast}(n))$ is equivalent to a new combinatorial problem called \textsc{SubsetConnectivity} (\Cref{thm:Equiv}).
For this problem, \yuki{we showed that at most one additional tree is needed if we increase~$n$ by one, i.e.,~$c(n+1) \in \{c(n),\, c(n)+1\}$ (\Cref{thm:cn-plus-one-or-cn}).
In the case that~$c(n+1) = c(n)$, we also showed that the solution to \textsc{SubsetConnectivity}$(n)$ must not contain an undirected star graph (\Cref{thm:nostar}).}

Finally, in \Cref{sec:Hybrid>cover}, we considered the version of \textsc{CoveringNumber}, where there is no restriction on the length of the sequence.
We used a topological ordering of the cherry order graph $G_T$ (\Cref{obs:cogAcyclic}) to show that the set of all non-trivial subtrees of a tree $T$ on $n$ leaves can be reduced by a single cherry-picking sequence of length exactly $\binom{n}{2}$.

\paragraph{Open problems.}
Our work suggests several natural directions for further study.
The most immediate open problem is to determine the covering number of general non-binary trees, beyond the surviving variant.
For this problem, we have only loose bounds for when the input is a star phylogenetic tree.
\taka{
In this setting, we have \rev{verified computationally} that~$c(n) =n-2$ whenever~$4\le n \le 7$ by using an exact algorithm based on a set cover solver.
We wonder if this equality holds for larger~$n$, or if the behavior changes for certain values of~$n$.}
Equivalently, we ask whether~$c(n+1) = c(n)$ for some~$n\ge 4$.

In another direction, it is natural to ask whether the covering number framework extends from trees to networks.
\textsc{CoveringNumber} has two inputs -- a network class and a family of subnetworks. in this paper, we considered trees and the set of all subtrees.
For future research, one can consider a generalization of either or both inputs.

We consider changing the first input from trees to networks.
This has partially been addressed in \Cref{sec:TCOrch}, where we asked the question for orchard networks and their subtrees and for tree-child networks and their subtrees. 
If~$N$ is an orchard, there may not be a cherry-picking sequence of~$N$ which reduces a given subtree (\Cref{fig:CPNTreeContainmentCounterExample}); if~$N$ is tree-child, there may not be a tree-child sequence of~$N$ which reduces a given subtree (\Cref{fig:TCNTreeContainmentCounterExample}).
For the former, it would be interesting to characterize orchard networks with finite covering number.
For the latter, one can ask if, for every subtree of~$N$, there exists a cherry-picking sequence of~$N$ which reduces it (since cherry-picking sequences are more general than tree-child sequences).
The example given in \Cref{fig:TCNTreeContainmentCounterExample} is not a counterexample to this question, as the sequence~$(2,3)(5,2)(3,2)(4,3)(1,2)(2,3)$ of~$N$ reduces~$T$.
Moreover, whenever the covering number for the network inputs is known to be finite, can we still compute it in polynomial time?

One can also modify the second input -- instead of the set of all subtrees, one can consider any family of subnetworks such as the set of all tree-child subnetworks or the set of all orchard subnetworks.
Here, we restrict to tree-child or orchard subnetworks as otherwise such subnetworks need not be reducible by cherry-picking sequences.

Finally, it would be of great interest to consider different combinations for the two inputs, to determine whether \textsc{CoveringNumber} remains tractable or becomes intractable for general choices of an input network and a family of subnetworks.

\section*{Acknowledgement}
We thank Marit Dee and Momoko Hayamizu for helpful discussions.

\section*{Data Availability}
No data were used in this research.

\printbibliography

\appendix
\section{Covering numbers for tree-child and orchard are non-finite}\label{sec:TCOrch}

In this section, we show that the covering number of a network may not be finite if we consider the classes of orchard and tree-child networks.
We start by stating definitions.


\taka{Let $N$ be a network, let $(x, y)$ be an ordered pair of distinct leaves of $N$, and let $p_x, p_y$ denote the parents of $x, y$, respectively. 
Cherries in a network and the reduction of a cherry are defined analogously to those in a tree (so as before, if~$p_x=p_y$,~$(x,y)$ is a cherry; see \Cref{sec:preliminaries}).
We call $(x,y)$ a \emph{reticulated cherry} of $N$ if $p_x$ is a reticulation, $p_y$ is a tree-vertex, and $p_y$ is a parent of $p_x$. 
We \emph{reduce} a reticulated cherry~$(x,y)$ in $N$ by deleting the edge~$(p_y, p_x)$ and suppressing any degree-2 vertices.
We call the resulting network~$N(x,y)$.
Let $S$ be a sequence of ordered pairs on distinct elements in $X$. 
Then~$NS$ is the network obtained by iteratively applying the elements of~$S$ to~$N$. 
We say that $S$ \emph{reduces} $N$ if $NS$ is a tree with a single leaf.
}

\taka{
A network $N$ on $X$ is \emph{orchard} if $N$ can be reduced by some cherry-picking sequence on $X$.
A cherry-picking sequence $S = (x_1,y_1)(x_2,y_2)\cdots(x_k,y_k)$ is \emph{minimal} for an orchard network $N$ if
\[N(x_1,y_1)(x_2,y_2)\cdots(x_{i-1},y_{i-1})\not\cong N(x_1,y_1)(x_2,y_2)\cdots(x_{i},y_{i})\]
for all~$i\in[|S|]$.
A network $N$ on $X$ is \emph{tree-child} if every non-leaf vertex of $N$ has at least one child which is a tree-vertex or a leaf.}

One can also characterize tree-child networks based on the type of cherry-picking sequences which reduce them.
We call a cherry-picking sequence $S$ a \emph{tree-child sequence (TCS)} on $X$ if the first coordinate of each ordered pair is not used as a second coordinate in the remainder of~$S$.
\taka{A network on~$X$ is tree-child if it can be reduced by some tree-child sequence on~$X$ (see Proposition 3 in~\cite{cherry_picking_and_network_containment}).}

\taka{
Similarly to the definition in \Cref{sec:preliminaries}, a set of minimal sequence $\mathcal{S}$ is said to \emph{cover} the set $\mathcal{N}$ of all subtrees of $N$ if each element of $\mathcal{N}$ is reduced by at least one element of $S$.
Then, it is natural to consider the covering numbers for tree-child networks and orchard networks as follows.
}
    Moreover, $S$ is called a \emph{tree-child sequence (TCS)} on $X$ if the first coordinate of each ordered pair is not used as a second coordinate in the remainder of~$S$.

\paragraph{Orchard and tree-child covering number.}
    For an orchard network $N$ on $X$, the \emph{orchard covering number} of $N$, written in $c_{or}(N)$, is the size of the smallest set of minimal CPSs of $N$ that covers the set of all subtrees of $N$. 
    For a tree-child network $N$ on $X$, the \emph{tree-child covering number} of $N$, written in $c_{tc}(N)$, is the size of the smallest set of minimal TCSs of $N$ that covers the set of all subtrees of $N$.



Interestingly, these numbers need not be finite.
It has been previously shown by \cite{cherry_picking_and_network_containment} that there is an orchard network that has a subtree on the same leaf set that is not reducible by any sequence of the network (see \Cref{fig:CPNTreeContainmentCounterExample}). It immediately follows from this result that the orchard covering number may not be finite for certain orchard networks.

\begin{observation}\label{obs:cor(N)infinite}
    Let~$N$ and~$T$ be the network and tree in \Cref{fig:CPNTreeContainmentCounterExample}, respectively.
    No CPS of~$N$ reduces~$T$.
    In particular, this means~$c_{or}(N)$ is not finite.
\end{observation}
\begin{proof}
    That no CPS of~$N$ reduces~$T$ follows from Theorem 4 of~\cite{cherry_picking_and_network_containment}.
\end{proof}

In a similar vein, the tree-child covering number may not be finite for certain tree-child networks (see \Cref{fig:TCNTreeContainmentCounterExample}).

\begin{observation}\label{obs:ctc(N)infinite}
    Let~$N$ and~$T$ be the network and tree in \Cref{fig:TCNTreeContainmentCounterExample}, respectively.
    No tree-child sequence of~$N$ reduces~$T$.
    In particular, this means~$c_{tc}(N)$ is not finite.
\end{observation}
\begin{proof}
    The first element of any TCS must be either $(2,3)$ or $(2,5)$. Then the leaf $2$ will form a cherry with $5$ or $3$, respectively. This means, $2$ must be picked as a first coordinate with either $5$ or $3$ in the rest of the sequence as a cherry. So $(1,2)$ or $(2,1)$ cannot appear in any TCS of $N$. However, in order for $T$ to be reduced, $(1,2)$ or $(2,1)$ must appear in the sequence.
\end{proof}


We note that Corollary 3 in~\cite{cherry_picking_and_network_containment} says that, for any tree-child network $N$, any subtree of $N$ with the same leaf-set is reduced by any TCS of $N$.
\Cref{obs:ctc(N)infinite} shows that the claim no longer holds when considering subtrees on different leaf-sets.

\end{document}